 \documentclass[reqno,11pt,a4paper]{amsart}

\usepackage{pgfplots}
\usepackage{subfigure}
\pgfplotsset{compat=1.12} 

\usepackage{amsfonts,amsmath,amssymb}
\usepackage[latin1]{inputenc}
\usepackage{dsfont}
\usepackage{enumerate}
\usepackage{xcolor}
\usepackage[all]{xy}

\def\notshow#1\notshowend{} %
\usepackage{hyphenat}

\usepackage{enumitem}
\usepackage{mathtools}

\usepackage{graphicx}

\usepackage{tikz}
\usepackage{tikz-3dplot}
\usepackage{pgfplots}

\usepackage{multirow}
\usepackage{array}

\newcommand{\C}{\mathcal{C}}

\newcommand{\df}{\mathrm{d}}
\usepackage{amssymb}
\usepackage{amsfonts}
\usepackage{mathrsfs}
\usepackage{hyperref}

\usepackage[normalem]{ulem} 

\def\bb#1\eb{\textcolor{blue}{#1}} 
\def\br#1\er{\textcolor{red}{#1}} %
\def\bm#1\em{\textcolor{purple}{#1}} %

\newcommand{\R}{\mathds R}

\newtheorem{thm}{Theorem}[section]
\newtheorem{prop}[thm]{Proposition}
\newtheorem{lemma}[thm]{Lemma}
\newtheorem{cor}[thm]{Corollary}
\theoremstyle{definition}
\newtheorem{defi}[thm]{Definition}
\newtheorem{notation}[thm]{Notation}

\newtheorem{rem}[thm]{Remark}

\newtheorem{conv}[thm]{Convention}

\newcommand{\ben}{\begin{enumerate}}
\newcommand{\een}{\end{enumerate}}
\newcommand{\bit}{\begin{itemize}}
\newcommand{\eit}{\end{itemize}}
\newcommand{\edoc}{\end{document}}

\title[Anisotropic rheonomic Huygens' 
principle]{Applications of cone structures to the anisotropic rheonomic Huygens' 
principle 
}

\author[M. A. Javaloyes]{Miguel \'Angel Javaloyes}\address{Departamento de Matem\'aticas, \hfill\break\indent Universidad de Murcia, \hfill\break\indent Campus de Espinardo,\hfill\break\indent 30100 Espinardo, Murcia, Spain} \email{majava@um.es}

\author[E. Pend\'as-Recondo]{Enrique Pend\'as-Recondo}\address{Departamento de Geometr\'{\i}a y Topolog\'{\i}a, Facultad de Ciencias, \hfill\break\indent Universidad de Granada,\hfill\break\indent Campus Fuentenueva s/n, \hfill\break\indent 18071 Granada, Spain}\address{Departamento de Matem\'aticas, \hfill\break\indent Universidad de Murcia, \hfill\break\indent Campus de Espinardo,\hfill\break\indent 30100 Espinardo, Murcia, Spain}\email{e.pendasrecondo@um.es}

\author[M. S\'anchez]{Miguel S\'anchez}\address{Departamento de Geometr\'{\i}a y Topolog\'{\i}a, Facultad de Ciencias, \hfill\break\indent Universidad de Granada,\hfill\break\indent Campus Fuentenueva s/n, \hfill\break\indent 18071 Granada, Spain}\email{sanchezm@ugr.es}

\begin{document}
\begin{abstract} 
A general framework for the description of classic wave propagation is introduced. This relies on a cone structure $\C$ determined by an intrinsic space $\Sigma$ of velocities of propagation 
(point, direction and time-dependent)   
and
an observers' vector field $\partial/\partial t$ whose integral curves provide both a Zermelo problem for the wave and an auxiliary Lorentz-Finsler metric $ G $ compatible with $\C$. 
The PDE for the wavefront is reduced to the ODE for the $t$-parametrized cone geodesics of $\C$. Particular cases  include time-independence ($\partial/\partial t$ is Killing for $G$), infinitesimally ellipsoidal propagation ($G$ can be replaced by a Lorentz metric) or the case of a medium which moves with respect to $\partial/\partial t$ faster than the wave (the ``strong wind'' case of a sound wave), where a conic time-dependent Finsler metric emerges. The specific case of wildfire propagation is revisited.

\end{abstract}

\maketitle


\vspace{-5mm}


\vspace{-4mm}

\tableofcontents
\section{Introduction}
Cone structures appear in different parts of Mathematics and they are the basis of Causality in standard Relativity as well as in recent extensions such as Finsler spacetimes (see \cite{JS20,Mak18,Minguzzi} and references therein). As pointed out by some authors  \cite{BLV05,DS19,Gib,GW11,Pal15}, the viewpoint of spacetimes can be used in non-relativistic settings to describe the propagation of certain physical phenomena that propagate through a medium at finite speed, e.g., wildfires or sound waves, and the framework can be extended to other phenomena such as seismic waves \cite{ABS03,BuSl05,Mon}, water waves, etc.
Indeed, this applies in some situations related to the classical Fermat's principle such as
 {\em Zermelo's navigation problem}, which  seeks the fastest trajectory between two prescribed points for a moving object with respect to a medium, which may also move with respect to the observer (see the recent detailed study in \cite{CJS, JS20}).  
Here, we will focus on 
 Huygens' (or, more properly, Huygens-Fresnel) principle, which states that every point on a wavefront at some instant is itself the source of secondary wavelets which determine the wavefront at later instants. Focusing on the wavefront, the cone structures allow one to consider  the most general situation where the velocity of propagation is anisotropic (i.e., direction-dependent and thus, non-spherical) and rheonomic (i.e., time-dependent). 
With minor modifications, the propagation of the wavefront of a wildfire (affected by the anisotropies of the ground and a possibly time-dependent wind) becomes an outstanding  example. This case was
  developed in the pioneering work by Markvorsen \cite{M16}, who showed the importance of Finslerian geometry for the modeling of wildfires \cite{M16} (simplifying the previous approach by Richards \cite{R}, see also \cite{In95,Teo52}) and introduced rheonomic Lagrangians which could be applied to this setting \cite{M17}. 
    Further works on Huygens' include Palmer \cite{Pal15}, where anisotropic wavefronts in a space endowed with a Minkowski norm are studied, and Dehkordi \& Saa, \cite{DS19}, where the time-independent Huygens' principle is studied in the context of Analogue Gravity (a research programme which investigates analogues of relativistic features within other physical systems \cite{BLV05}),
following the line of Zermelo's problem and wildfire spreading in \cite{M16}.
    
   Here, we will go beyond in the geometric interpretations and will generalize the setting by showing:
\begin{enumerate}
\item The abstract theory of cone structures establishes a general geometric framework to model waves, \S \ref{sec:general_setting}.
Indeed, the wave propagation velocity $\Sigma$ at each point, direction and instant of time provides the cone structure $\C$ on a manifold $M=\R\times N$ with a natural time coordinate $t: M\rightarrow \R$. Then, starting from the initial source $ S_0 $ (say, a compact submanifold), its causal future $ J^+(S_0) $ for $\C$ provides the region that will be affected by the wave. Moreover, the wavefront at each instant $ t_0 $ is given by the boundary $ \partial J^+(S_0)$ intersected with the slice $ \{t=t_0\} $. In the case of wildfires, $S_0$ becomes the boundary of the initial burned area $B_0$ and the wave propagation towards the interior of $B_0$ is neglected (see Fig. \ref{fig:cones}).

\item The triple $(\Omega= dt,\partial/\partial t,\Sigma)$ on $M$ not only characterizes $\C$ but also provides the spacetime trajectories  to be reached by the wavefront (namely, the integral curves of  the observers' vector field $T=\partial/\partial t$), thus, providing a link with Zermelo's problem, \S \ref{s4.1}.

\item The wavefronts are characterized by  the geodesics of the cone structure $\C$, \S \ref{subsec:cone_geod}, which are governed by an ODE system rather than a PDE one, \S  \ref{subsec:orth_cond}. Such cone geodesics $\hat\gamma$ of $ M $ (parametrized by $t\in \R$) represent the spacetime wave trajectories that arrive first at the integral curves of $ \partial/\partial t $. So, their projections $\gamma$ on $ N $ represent the fastest wave propagation trajectories through the space $ N $.
Cone geodesics can always be interpreted as the lightlike geodesics (up to a suitable reparametrization) of a  Lorentz-Finsler metric $G$ which can be canonically chosen from $(dt,\partial/\partial t,\Sigma)$. 
Thus, a neat geometric interpretation of the evolution of the wavefront is obtained. 
 
 \item The  intrinsic/extrinsic  character (with respect to the propagation of the wave) of the introduced  elements  yields relevant geometric consequences. As the wave propagation  itself is not relativistic, the time $t$ is regarded as an absolute coordinate and both, the cone structure $\C$ and the space of velocities $\Sigma$ can be regarded as intrinsic objects in $ M $. 
However, the observers' vector field $T=\partial/\partial t$ is extrinsic to the wave, and different choices will provide different splittings $M=\R\times N$ and Zermelo's problems. In particular,\footnote{It is worth pointing out that the following two items apply to the interpretation of Kerr spacetime and other black holes metrics in Analogue Gravity \cite{GHWW, HL}. Here $T=\partial/\partial t$ is a Killing vector field  which becomes timelike far from the black hole but spacelike in its ergosphere and inside the event horizon.}

\bit\item An intrinsic property that $\C$ might satisfy is to be compatible with a classsical (non-Finslerian) Lorentz metric $g$. Noticeably, this is the case studied by Richards \cite{R} and Markvorsen \cite{M16}, the latter introducing a Finsler metric $F$. However, such an $F$ can be skipped, since an alternative description of the problem (including the ODE's for cone geodesics) follows in terms of $g$, \S \ref{s5}.
 \item The observers' vector field $T=\partial/\partial t$ will be used to determine a canonic Lorentz-Finsler metric $G$ but, depending on the problem, more than one choice of $T$ might be interesting. This may be especially relevant when the speed of the wave with respect to the medium is smaller than the speed of the medium with respect to the observers (``strong wind'' case).  Cone geodesics depend only on $\C$ but a choice  of observers $T^0$ comoving with the medium might be convenient. In such case, the arrival time of the wave with respect to the observers of $T$ and $T^0$ differs. We will focus extensively on the (``mild wind'') case when $T$ is a $\C$-timelike vector field, so that one can write $G=dt^2-F^2$, being $F$ a (time-dependent) Finsler metric, but this restriction will also be removed, \S \ref{s6}.
 \eit
\een

 The paper is structured as follows. In \S \ref{s2}, the necessary background on cone structures and Lorentz-Finsler metrics is introduced, following mainly \cite{JS20}. 
In \S \ref{sec:general_setting}, the wave propagation on  $M=\R\times N$ is  modelled heuristically  along the first three subsections. Focusing on the mild wave case, we start with $\partial/\partial t$ and $\Sigma$ and arrive at a cone triple
$(dt,\partial/\partial t,F)$  with the corresponding Lorentz-Finsler metric $G$ (the setup and choices are summarized in Conv. \ref{convention}). The time-independent case corresponds to identifying $F$ with a Finsler metric on $N$ and $\partial/\partial t$ being a (timelike) Killing vector field, so that $(M,G)$ is a stationary Finsler spacetime (see \cite[\S 4.2]{JS20}).
The paradigmatic case of wildfires is detailed in \S \ref{subsec:wildfires}.

In \S \ref{s4}, the wavefront is computed. First, the {\em wavemap} is introduced in the spacetime and the wavefront at each time is interpreted there, \S \ref{s4.1}. The spacetime trajectories of the wave are characterized as lightlike pregeodesics of $G$, which represent the locally fastest trajectories from $S_0$, i.e., the first-arriving ones until the null cut points from $S_0$. The key Thm. \ref{cor:zermelo} shows that the corresponding null cut function is positively lower bounded on $S_0$. As $S_0$ has arbitrary codimension, the main difficulty comes from the fact that the space of orthogonal vectors to the submanifold at every point is not linear, but a conic submanifold with a singularity in the origin 
(see also the discussion below Cor. \ref{cor:zermelo0} for other subtleties on the cut function). This problem is solved in Lem. \ref{th:zermelo} by introducing a suitable family of hypersurfaces $S^w$ which contain $ S $ for every orthogonal vector $w$ close to a prescribed one $v$. The proof that orthogonal geodesics solve Zermelo's problem to $S^w$ (and then to $ S $) is now approachable, since the space of orthogonal vectors to $S^w$ consists of two linear vector bundles, where one can define a smooth exponential map. Observe that the proof combines causal and Finsler tools, being both necessary. Rem. \ref{r_resumen_s4.2} summarizes this part. In \S \ref{subsec:cone_geod} we find the equation of the spatial trajectories by obtaining first the equation of the $G$-geodesics and modifying this equation so that the suitable reparametrization is achieved. So, an ODE solution for the computation of the wavefront is achieved, 
Thm. \ref{th:geod_gen}. In \S \ref{subsec:orth_cond}, the results are particularized to the case of wildfire propagation. Taking into account our ODE solution, a PDE one is also obtained  from the spacetime viewpoint (Thm. \ref{th:ort_cond}). This allows us to revisit the PDE solution obtained by Markvorsen \cite{M16,M17} from a more classical Lagrangian viewpoint.

In \S \ref{s5} we consider the case when the space $\Sigma_p$ of velocities at each $p$ is an ellipsoid. According to \cite{M16,R}, this is  an experimental fact for wildfires; anyway, it can be regarded as a first approximation to anisotropies. If the ellipsoid is centered at the origin then $F$ becomes a (time-dependent) Riemannian metric $h$ but, otherwise, such centers determine a vector field $W$ and, then, a Finsler metric of Randers type with Zermelo data $(h,W)$, as studied in \cite{M16}. Here, we observe that this ellipsoidal case is also the case when $\C$ is the cone structure of a Lorentzian metric  $g$. Then, this metric and the corresponding ODE for cone geodesics are written explicitly in Prop. \ref{lem:lorentz} and Thm. \ref{t5.3}, resp., providing a non-Finslerian alternative to \cite{M16}. Moreover, the latter theorem yields an extension to the time-dependent case of the  correspondence
between relativistic stationary spacetimes and 
 Randers spaces in \cite{CJS11} (compare also with \cite{GHWW}).
 
In \S \ref{s6}, the case of strong wind is considered. Now, $\Sigma$ and $T=\partial/\partial t$ yield a (time-dependent) wind Finslerian structure, \S \ref{6.0}. This is a geometric notion introduced  in  \cite{CJS} with independent interest (see, e.g., \cite{JS17}). In \S \ref{6.1} we explain how $\C$ and the region affected by the wave can be described as before just replacing $T$ by a comoving vector field. However, as emphasized in Rem. \ref{r6.1} the trajectories of the wave are naturally described in terms of the conic Finsler metric associated with the wind Finslerian structure. Finally, an application to the active part of a wildfire is also given, \S \ref{6.2}.

Summing up, the spacetime viewpoint we adopt throughout this work enables us to settle a unified geometric framework for wave propagation in general situations (time and direction-dependent propagation, arbitrary wind, and space and source of arbitrary dimensions). As a consequence, we have been able to deal with new issues such as the null cut properties of the wave trajectories, \S \ref{s4.1}, or the strong wind case, \S \ref{s6}.

\section{Preliminaries on Lorentz-Finsler cones}\label{s2}
In order to make this work as self-contained as possible, we summarize in this section the main definitions and results we will use regarding cone structures and Lorentz-Finsler metrics, following \cite{JS20}.

Throughout this section, $ V $ and $ M $ will denote a real vector space and a smooth (namely, $C^\infty$) manifold, resp., of dimension $ m = n+1 \geq 3 $, being $ TM $ the tangent bundle of $ M $.

\subsection{Cone structures and causality}
We start by introducing the definition of cones at the level of vector spaces. This notion, when transplanted to manifolds, will generate what we will call a cone structure.
\begin{defi}
A smooth hypersurface $ \mathcal{C}_0 $ embedded in $ V \setminus \lbrace 0 \rbrace $ is a {\em cone} if it satisfies the following properties:
\begin{enumerate}
\item {\em Conic}: for all $ v \in \mathcal{C}_0 $, $ \lbrace \lambda v: \lambda > 0 \rbrace \subset \mathcal{C}_0 $.
\item {\em Salient}: if $ v \in \mathcal{C}_0 $, then $ -v \notin \mathcal{C}_0 $.
\item {\em Convex interior}: $ \mathcal{C}_0 $ is the boundary in $ V \setminus \lbrace 0 \rbrace $ of an open subset $ A_0 \subset V \setminus \lbrace 0 \rbrace $ (the $ \mathcal{C}_0 $-interior) which is convex, in the sense that, for any $ v,u \in A_0 $, the segment $ \lbrace \lambda v + (1-\lambda)u: 0 \leq \lambda \leq 1 \rbrace \subset V $ is included entirely in $ A_0 $.
\item {\em (Non-radial) strong convexity}: the second fundamental form of $ \mathcal{C}_0 $ as an affine hypersurface of $ V $ is positive semi-definite (with respect to an inner direction pointing out to $ A_0 $) and its radical at each point $ v \in \mathcal{C}_0 $ is spanned by the radial direction $ \lbrace \lambda v: \lambda > 0 \rbrace $.
\end{enumerate}
\end{defi}

Any cone can be constructed by taking a compact strongly convex hypersurface $\Sigma_0$ of an  affine hyperplane $\Pi\subset  V$, with $0\not\in \Pi$, and taking all the open half-lines through $\Sigma_0$ starting at $0$ \cite[Lem. 2.5]{JS20}.

\begin{defi}
\label{def:cone_structure}
A {\em cone structure} $ \mathcal{C} $ is an embedded hypersurface of $ TM $ such that, for each $ p \in M $:
\begin{enumerate}
\item $ \mathcal{C} $ is transverse to the fibers of the tangent bundle, i.e., if $ v \in \mathcal{C}_p \coloneqq T_pM \cap \mathcal{C} $, then $ T_v(T_pM) + T_{(p,v)}\mathcal{C} = T_{(p,v)}(TM) $, and
\item $ \mathcal{C}_p $ is a cone in $ T_pM $.
\end{enumerate}
We denote by $ A_p $ the $ \mathcal{C}_p $-interior, and $ A \coloneqq \cup_{p \in M} A_p $.
\end{defi}

Notice that,  even if $\C$ is smooth, the transversality condition (1) is necessary to ensure that the fibers $\C_p$ vary smoothly with $p\in M$.

A cone structure provides some classes of privileged vectors, which can be used to define the usual notions about causality.

\begin{defi}
Given a cone structure $ \mathcal{C} $ in $ M $, we say that a vector $ v \in T_pM $ is
\begin{itemize}
\item {\em timelike} if $ v $ or $ -v $ belongs to $ A_p $,
\item {\em lightlike} if $ v $ or $ -v $ belongs to $ \mathcal{C}_p $,
\item {\em causal} if it is timelike or lightlike, i.e., if $ v $ or $ -v $ belongs to $ \overline{A}_p \setminus \lbrace 0 \rbrace $,
\item {\em spacelike} if it is not causal.
\end{itemize}
Also, a causal vector is {\em future-directed} if $ v \in \overline{A}_p \setminus \lbrace 0 \rbrace $ and {\em past-directed} if $ -v \in \overline{A}_p \setminus \lbrace 0 \rbrace $.

Analogously, we say that a piecewise smooth curve $ \gamma: I \rightarrow M $ is future-directed (resp. past-directed) timelike, lightlike or causal, when its tangent vector $ \gamma' $ (or both $ \gamma'(t_0^+) $ and $ \gamma'(t_0^-) $ at any break $ t_0 \in I $) is future-directed (resp. past-directed) timelike, lightlike or causal.

Moreover, we say that two points $ p,q \in M $ are {\em chronologically related}, denoted $ p \ll q $, if there exists a future-directed timelike curve from $ p $ to $ q $, and {\em causally related}, denoted $ p \leq q $, if either $ p = q $ or there exists a future-directed causal curve from $ p $ to $ q $. This allows us to define the following sets:
\begin{itemize}
\item {\em chronological future}: $ I^+(p) \coloneqq \lbrace q \in M: p \ll q \rbrace $,
\item {\em chronological past}: $ I^-(p) \coloneqq \lbrace q \in M: q \ll p \rbrace $,
\item {\em causal future}: $ J^+(p) \coloneqq \lbrace q \in M: p \leq q \rbrace $,
\item {\em cusal past}: $ J^-(p) \coloneqq \lbrace q \in M: q \leq p \rbrace $,
\end{itemize}
and the {\em horismotic relation}: $ p \rightarrow q $ when $ q \in J^+(p) \setminus I^+(p) $. 

Finally, a {\em time function} is a real function $ t: M \rightarrow \R $ which is strictly increasing when composed with future-directed timelike curves. In addition, if $ t $ is also smooth and no causal vector is tangent to the slices $ \lbrace t = \textup{constant} \rbrace $, then it is called a {\em temporal function}.
\end{defi}

\begin{rem}
When we consider a classical Lorentzian metric $ g $ on $ M $, its lightlike vectors (those nonzero vectors $ v \in TM $ that verify $ g(v,v) = 0 $) provide globally two cone structures, one future-directed and the other one past-directed (see \cite[Cor. 2.19]{JS20}). Therefore, the notions defined above trivially generalize those in the Causal Theory of classical spacetimes.
\end{rem}

Cone structures also admit the notion of geodesic.

\begin{defi}
Let $ \mathcal{C} $ be a cone structure. A continuous curve $ \gamma: I \rightarrow M $ is a {\em cone geodesic} if it is locally horismotic, i.e., for each $ t_0 \in I $ and any neighborhood $ V $ of $ \gamma(t_0) $, there exists a smaller neighborhood $ U   \subset V $ of $ \gamma(s_0) $ such that, if $ I_{\varepsilon} \coloneqq [t_0-\varepsilon,t_0+\varepsilon] \cap I $ satisfies $ \gamma(I_{\varepsilon}) \subset U $ for some $ \varepsilon > 0 $, then
$$
t_1 < t_2 \Leftrightarrow \gamma(t_1) \rightarrow_U \gamma(t_2), \quad \forall t_1,t_2 \in I_{\varepsilon},
$$
where $ \rightarrow_U $ is the horismotic relation for the natural restriction $ \mathcal{C}_U $ of the cone structure to $ U $.
\end{defi}

\subsection{Lorentz-Finsler metrics and Finsler spacetimes}
All variants of Minkowski norms are introduced at the level of vector spaces, from which the notions of Lorentz-Finsler and Finsler metrics on a manifold appear.

\begin{defi}
A positive function $ L: A_0 \subset V \setminus \lbrace 0 \rbrace \rightarrow \R^+ $ is a (proper) {\em Lorentz-Minkowski norm} if
\begin{enumerate}
\item $ A_0 $ is a conic domain (i.e., $ A_0 $ is open, non-empty, connected and if $v\in A_0$, then $\lambda v\in A_0, \forall \lambda>0$),
\item $ L $ is smooth and positively two-homogeneous, i.e., $ L(\lambda v) = \lambda^2L(v) $ for all $ v \in A_0, \lambda > 0 $,
\item for every $ v \in A_0 $, the {\em fundamental tensor} $ g_v $, given by
\begin{equation}
\label{eq:fund_tensor}
g_v(u,w) = \left.\frac{1}{2}\frac{\partial^2}{\partial\delta\partial\eta}L(v + \delta u + \eta w)\right\rvert_{\delta=\eta=0}, \quad \forall u,w \in V,
\end{equation}
has index $ m-1 $, and
\item the topological boundary $ \mathcal{C}_0 $ of $ A_0 $ in $ V \setminus \lbrace 0 \rbrace $ is smooth and $ L $ can be smoothly extended as zero to $ \mathcal{C}_0 $ with non-degenerate fundamental tensor.
\end{enumerate}
\end{defi}

\begin{defi}
Let $ A \subset TM \setminus \textbf{0} $ be a conic domain (i.e., each $ A_p \coloneqq A \cap T_pM $ is a conic domain of $ T_pM \setminus \{0\} $ for all $ p \in M $) such that its closure in $ TM \setminus \textbf{0} $ is an embedded smooth manifold with boundary. Let $ \mathcal{C} \subset TM \setminus \textbf{0} $ be its boundary and $ L: A \rightarrow \R^+ $ a smooth function which can be smoothly extended as zero to $ \mathcal{C} $ satisfying, for all $ p \in M $, that $ L_p \coloneqq L\rvert_{A_p} $ is a (proper) Lorentz-Minkowski norm. Then, $ L $ will be called a (proper) {\em Lorentz-Finsler metric} on $ M $, and $ (M,L) $ a {\em Finsler spacetime}. When necessary, $ L $ will be assumed continuously extended to the zero section $ \textbf{0} \subset TM $.
\end{defi}

Each Lorentz-Finsler metric determines a unique cone structure. More precisely \cite[Cor. 3.7]{JS20}:

\begin{prop}
\label{prop:cone_lf}
If $ L: A \rightarrow \R^+ $ is a Lorentz-Finsler metric, then the boundary $ \mathcal{C} $ of $ A $ in $ TM \setminus \textbf{0} $ is a cone structure with cone domain $ A $. $ \mathcal{C} $ will be called {\em the cone structure of} $ L $.
\end{prop}

As explained below  (Prop. \ref{p_anisotrop}), each cone structure can be obtained from such an $L$ in a highly non-unique way.

\begin{rem}
Since a Lorentz-Finsler metric $ L $ is smooth on $ \mathcal{C} $ with non-degenerate fundamental tensor, $ L $ can be smoothly extended to an open conic subset $ A^* $ containing $ \overline{A} \setminus \lbrace 0 \rbrace $ such that the fundamental tensor of $ L $ has index $ m-1 $ on $ A^* $ and $ L < 0 $ in $ A^* \setminus \overline{A} $. Clearly, such an $ A^* $ can be chosen as a conic domain.
\end{rem}

The required two-homogeneity of Lorentz-Finsler metrics is due to the lack of differentiability of one-homogeneous functions on lightlike vectors. However, we now recover the notion of classical Finsler metrics, for which it is more convenient to choose one-homogeneous functions.

\begin{defi}
A {\em conic Minkowski norm} (resp. {\em Lorentzian norm}) is a smooth positive function $ F: A_0 \subset V \setminus \lbrace 0 \rbrace \rightarrow \R^+ $, being $ A_0 $ a conic domain and $ F $ positively one-homogeneous, satisfying that the fundamental tensor $ g_v $ in \eqref{eq:fund_tensor} for $ L = F^2 $ is positive definite (resp. has index $ m-1 $) for all $ v \in A_0 $. In addition, when $ A_0 = V \setminus \lbrace 0 \rbrace $ we say that $ F $ is a {\em Minkowski norm}.
\end{defi}

\begin{defi}
A {\em conic Finsler metric} (resp. {\em Lorentzian Finsler metric}) on $ M $ is a smooth function $ F: A \subset TM \setminus \textbf{0} \rightarrow \R^+ $, being $ A $ a conic domain, such that $ F_p \coloneqq F\rvert_{A_p} $ is a conic Minkowski norm (resp. Lorentzian norm) for all $ p \in M $. When $ A = TM \setminus \textbf{0} $ (so that $ F_p $ is a Minkowski norm for all $ p \in M $), we say that $ F $ is a {\em Finsler metric}. When convenient, Finsler metrics are assumed continuously extended to the zero section.
\end{defi}

This definition is trivially extended to any vector bundle $ VM $ (in particular, to any subbundle of $ TM $) in such a way that a (conic) Finsler metric (resp. Lorentzian Finsler metric) on $ VM $ becomes a smooth distribution of (conic) Minkowski norms (resp. Lorentzian norms) in each fiber of the bundle.

\begin{defi}
\label{def:orth}
Let $ L: A^* \rightarrow \R $ be a Lorentz-Finsler metric on $ M $ with fundamental tensor $ g $. For any $ v \in A^*_p,w \in T_pM $, we say that $ v $ is $ L $\textit{-orthogonal} to $ w $, denoted $ v \perp_L w $, if
\begin{equation}\label{e3}
g_v(v,w) = \left.\frac{1}{2} \frac{d}{d \delta} L(v+\delta w)\right\rvert_{\delta = 0} = 0.
\end{equation}
Analogously, $ v \in A_p $ is $ F $\textit{-orthogonal} to $ w \in T_pM $ for a (conic or Lorentzian) Finsler metric $F:A \rightarrow \R^+$, $ v \perp_F w $, if \eqref{e3} holds with $L=F^2$. Also, we say that $ v \in A^*_p $ (resp. $ v \in A_p $) is $ L $-orthogonal (resp. $ F $-orthogonal) to a submanifold $ S \subset M $, with $ p \in S $, if $ v \bot_L w $ (resp. $ v \bot_F w $) for all $ w \in T_pS $.
\end{defi}

\subsection{Cone triples}\label{s2.conetriples}
Cone structures can be univocally determined by a triple that includes a Finsler metric, providing then a natural link between both notions (\cite[Lem. 2.15, Thm. 2.17]{JS20}).

\begin{lemma}
Given a cone structure $ \mathcal{C} $, one can find on $ M $:
\begin{itemize}
\item[(i)] a timelike one-form $ \Omega $ (i.e., $ \Omega(v) > 0 $ for any future-directed causal vector $ v $),
\item[(ii)] an $ \Omega $-unit timelike vector field $ T $ ($ T $ is timelike and $ \Omega(T) = 1 $).
\end{itemize}
\end{lemma}

\begin{rem}
The one-form $ \Omega $ is neither exact nor closed in general, but locally it can be chosen exact, so that $ \Omega = dt $ for some smooth function $ t: U \subset M \rightarrow \R $ (see \cite[Rem. 2.16]{JS20}). In this case, $ t $ is naturally a temporal function for the restriction $ \mathcal{C}_U $ of the cone structure to $ U $.
\end{rem}

Any pair $ (\Omega,T) $ associated with $ \mathcal{C} $ (in the sense of the previous lemma) yields a natural splitting $ TM = \textup{Span}(T) \oplus \textup{Ker}(\Omega) $ with the projection $ \pi^{\Omega}: TM \rightarrow \textup{Ker}(\Omega) $ determined by
\begin{equation}
\label{eq:pi_omega}
v = \Omega(v)T_p + \pi^{\Omega}(v), \quad \forall v \in T_pM, p \in M.
\end{equation}

\begin{thm}
\label{th:cone_triple}
Let $ \mathcal{C} $ be a cone structure. For any choice of a timelike one-form $ \Omega $ and an $ \Omega $-unit timelike vector field $ T $, there exists a unique Finsler metric $ F $ on the vector bundle $ \textup{Ker}(\Omega) \subset TM $ such that, for any nonzero $ v \in T_pM $, $ p \in M $,
\begin{equation}
\label{eq:cone_triple}
v \in \mathcal{C} \Leftrightarrow v = F(\pi^{\Omega}(v))T_p + \pi^{\Omega}(v).
\end{equation}
Moreover, the indicatrix of $ F $ is $ \Sigma = \pi^{\Omega}(\Omega^{-1}(1) \cap \mathcal{C}) $.

Conversely, for any cone triple $ (\Omega,T,F) $ composed by a non-vanishing one-form $ \Omega $, an $ \Omega $-unit vector field $ T $ and a Finsler metric $ F $ on $ \textup{Ker}(\Omega) $, there exists a unique cone structure $ \mathcal{C} $ satisfying \eqref{eq:cone_triple}, which will be said {\em associated with the cone triple}.
\end{thm}

There is a particular Lorentz-Finsler metric associated with a given cone structure that will be very useful along this work for its simplicity.

\begin{prop}
For any cone triple $ (\Omega,T,F) $ with associated cone structure $ \mathcal{C} $, the continuous function $ G: TM \rightarrow \R $ defined by $ G \coloneqq \Omega^2-F^2 $, i.e.,
$$
G(\tau T_p+v) = \tau^2-F(v)^2, \quad \forall \tau \in \R, v \in \textup{Ker}(\Omega_p), p \in M,
$$
is smooth on $ TM \setminus \textup{Span}(T) $. Moreover, whenever it is smooth, its fundamental tensor (computed as in \eqref{eq:fund_tensor}) is non-degenerate with index $ m-1 $.
\end{prop}

\begin{rem}
\label{rem:lf_metric}
Although $ G $ is not properly a Lorentz-Finsler metric because it fails to be smooth on $ \textup{Span}(T) $ (unless $ F $ is Riemannian), it can be shown that for any neighborhood $ U $ of the section $ T $ (regarded as a submanifold of $ TM $) there exists a proper Lorentz-Finsler metric $ L $ defined on all $ TM $ such that $ L = G $ in $ TM $ away from $ U $ (see \cite[Thm. 5.6]{JS20}). As a consequence, any cone structure $ \mathcal{C} $ is the cone structure of a (smooth) Lorentz-Finsler metric defined on all $ TM $.

Nonetheless, for the purposes of this work (in which we will only be interested in lightlike curves), we only need the Lorentz-Finsler metric to be smooth on a neighborhood of $ \mathcal{C} $, so we can always use $ G $ as the Lorentz-Finsler metric associated with a given cone structure (see \cite{CS,LPH} for other works using $ G $). Observe also that the relation between the fundamental tensor of $ G $, $ g^G $, and the fundamental tensor of $ F $, $ g^F $, is
\begin{equation}
\label{eq:rel_g_gF}
g^G_v(u,w) = \Omega(u)\Omega(w)-g^F_{\pi^{\Omega}(v)}(\pi^{\Omega}(u),\pi^{\Omega}(w)),
\end{equation}
for all $ v \in T_pM \setminus \textup{Span}(T_p) $, $ u,w \in T_pM $, $ p \in M $.
\end{rem}

The discussion above underlies the following converse to Prop. \ref{prop:cone_lf} (see \cite[Rem. 5.9]{JS20}).

\begin{prop}\label{p_anisotrop}
Each cone structure $ \mathcal{C} $ uniquely determines a (non-empty) class of Lorentz-Finsler metrics.\footnote{Two metrics $ L_1,L_2$ sharing the same $\C$ are called {\em anisotropically equivalent}. In this case, there is a  smooth positive function $ \mu: \overline{A} \setminus \textbf{0} \rightarrow \R^+ $ such that $ L_2 = \mu L_1 $ (see \cite[Thm. 3.11]{JS20}).}
\end{prop}

Moreover, all these Lorentz-Finsler metrics share the same lightlike pregeodesics,\footnote{Geodesics are smooth autoparallel curves for the Chern connection of $L$. Recall that $L$ can be extended to some conic domain $ A^* $ which includes $ \overline{A} \setminus \textbf{0}$. The Chern connection can be defined on all $A^*$ and it is uniquely determined on $A$.} which coincide with those of $\C$ (\cite[Thm. 6.6]{JS20}):

\begin{thm}
A  curve $ \gamma: I \rightarrow M $ is a cone geodesic for a cone structure $ \mathcal{C} $ if and only if $ \gamma $ is a lightlike pregeodesic for one (and then, for all) Lorentz-Finsler metric $ L $ with cone structure $ \mathcal{C} $.
\end{thm}

Observe that the focal points of these lightlike geodesics are also preserved in the whole anisotropically conformal class \cite{JavSoa20} and these geodesics coincide with those of \cite{Mak18}.

\section{General setting and Huygens' principle}
\label{sec:general_setting}

\subsection{Basics of the model} \label{subsec:general_setting}
Let us start with some notation for modeling the propagation of an anisotropic wave. The space where the wave propagates can be an arbitrary smooth manifold $ \hat{N} $ of dimension $ n \geq 2 $. Nevertheless, global properties will be easy to deduce from the local ones, so for computations we can assume that there is a global chart $\phi$ and work with $ N\coloneqq\phi(\hat N) \subset \R^n $ in natural coordinates $x=(x^1,\ldots,x^n) $. To include the (non-relativistic) time $t$ in the model, we define the {\em spacetime} $ M \coloneqq \R \times N $, being $ t: M \rightarrow \R $ the natural projection. Other useful natural projections will be $ \pi^M: TM \rightarrow M $, $ \pi^N: TN \rightarrow N $ and $ \pi_N: M \rightarrow N $.

At each point $ p = (t,x) \in M $ the wave propagates in all spatial directions, although, in general, its velocity may vary from one direction to another. Mathematically, we will assume that the propagation of the wave at each $ p \in M $ is given by a (strongly convex) oval $ \Sigma_p $\footnote{This is, $ \Sigma_p $ is diffeomorphic to a sphere and its second fundamental form with respect to one (and then all) transversal vector field is (positive or negative) definite.} on the vector space $ \textup{Ker}(dt_p) = T_p(\{t\}\times N) \equiv T_xN $. This means that a vector $ v \in \textup{Ker}(dt_p) $ represents the velocity of the wave in the spatial direction determined by $ v \in T_xN $ at the time $ t \in \R $ if and only if $ v \in \Sigma_p $. If the zero vector lies in the open region enclosed by $\Sigma_p$, then $ \Sigma_p $ defines a  Minkowski norm $ F_p $ on $ \textup{Ker}(dt_p) $, being $ \Sigma_p $ its indicatrix (see \cite[Thm. 2.14]{JS14}). In this case, a hypersurface $ \Sigma \subset \text{Ker}(dt) $ of such ovals $ \Sigma_p $ varying smoothly with $ p \in M $ (in the sense that $ \Sigma $ is transverse to the fibers of $ \text{Ker}(dt) $, as in Def. \ref{def:cone_structure}) determines a Finsler metric $ F $ on the vector bundle $ \textup{Ker}(dt) \subset TM$, whose unit vectors represent the velocities of the wave on $ N $. For the convenience of the reader, we will restrict ourselves to this case. However, we will see in $\S \ref{s6}$ that the approach can be easily extended if the unit ball does not contain the zero vector at some points.
 
Next, our aim will be to determine the propagation of the wave assuming that $ \Sigma $ (and hence $ F $) is known. 

\subsection{Cone structure}
The previous elements allow one to introduce the cone triple $(\Omega \coloneqq dt,T \coloneqq \partial/\partial t,F) $ in $ M = \R \times N $ (so that $ \Sigma = (\mathcal{C} \cap dt^{-1}(1))-\frac{\partial}{\partial t} $) and, thus, the corresponding cone structure $\C$
and Lorentz-Finsler metric $ G = dt^2-F^2 $, according to \S \ref{s2.conetriples}.

\begin{notation}
\label{notacion}
In general, a vector $ \hat{v} \in T_pM = \text{Span}(\frac{\partial}{\partial t}\rvert_p) \oplus \text{Ker}(dt_p) $ will be written in the form (recall \eqref{eq:pi_omega})
$$
\hat v = \tau \left.\frac{\partial}{\partial t}\right\rvert_p + v \equiv (\tau,v), \text{ with } v=\pi^{\Omega}(\hat{v}) \in \text{Ker}(dt_p) (\equiv T_xN).
$$
Notice the last natural identification $ \text{Ker}(dt_p)\equiv T_xN $, to be used when convenient.
\end{notation}

In a natural way, a particle will be represented by a curve $\hat\gamma$ in the spacetime $M$ parametrized by the time $t$. By construction, the particle moves at the same speed as the wave (at each space point, instant and direction) if and only if $\hat\gamma(t)=(t,\gamma(t))$ is $\C$-lightlike. Indeed, from \eqref{eq:cone_triple}, $ \hat\gamma'(t) = (1,\gamma'(t)) $ is lightlike if and only if $ \gamma'(t) $ is $ F $-unit, i.e., $ \gamma'(t) \in \Sigma_p $ ($ \gamma'(t) $ coincides with the velocity of the wave through the space $ N $).

As happens for any cone triple, no causal vector is tangent to the slices $ \lbrace t=t_0 \rbrace \coloneqq \lbrace t_0 \rbrace \times N$. This means that, for any future-directed timelike curve $ \hat{\gamma}$, the composition $ t \circ \hat{\gamma} $ is strictly increasing and $ t: M \rightarrow \R $ is a temporal function. The intrinsic properties of $\C$ are independent of the selected cone triple. Nevertheless, it is necessary to work specifically with $ (dt,\partial/\partial t,F) $ here because it settles the time flow, the slices at constant time, etc. Physically, the choice of this specific cone triple means that $ \Sigma_p $ is the {\em infinitesimal wavefront}, i.e., the wavefront after one time unit for a wave starting at the origin of the tangent space $ \textup{Ker}(dt_p) $, assuming that the initial conditions at $ p $ remain constant.

\subsection{Wavefronts}
The wavefront at each instant of time will be given by the generalized time-dependent version of Huygens' envelope principle, which we call {\em anisotropic rheonomic Huygens' principle}: each point of the wavefront front$({t_0})$ at time $ t = t_0 $ becomes the source of a secondary wave, so that the wavefront front$({t_1})$ at $ t = t_1 $ is the envelope of these secondary waves of lapse $ t_1-t_0 $.

Since we have established that particles moving at the same speed as the wave are lightlike curves, the spatial points that can be reached by a wave starting at $ p_0 = (t_0,x_0) \in M $ are the projections of those in the causal future $ J^+(p_0) $, and thus the wavefront generated by a single point after a lapse $ t_1-t_0 $ consists of the outermost points reached by the wave at $ t = t_1 $, i.e., $ \partial J^+(p_0) \cap \lbrace t=t_1 \rbrace $. Therefore, Huygens' principle implies
\begin{equation}\label{e_wavefront_hatf}
\hbox{front}({t_1}) = \partial \left(\cup_{p\in \hbox{{\small front}}({t_0})} J^+(p)\right) \cap \lbrace t=t_1 \rbrace = \partial J^+({\mbox{front}({t_0})}) \cap \lbrace t=t_1 \rbrace
\end{equation}
and we will know how the wave expands over time by finding $ \partial J^+(\mbox{front}({t_0}))$.

\begin{rem}
Technically, we can only ensure that the wave reaches the boundary when the spacetime $ M $ is {\em causally simple}.\footnote{In our setup, this means that all $ J^{\pm}(p) $ are closed (thus, equal to the closure  $\overline{ I^{\pm}(p)}$).} Otherwise, the wave might not reach $ \partial J^+(\mbox{front}({t_0})) \cap \lbrace t=t_1 \rbrace $ but only remain arbitrarily close (anyway, it should be regarded  as the wavefront front$(t_1)$ even in this case). However, in the cases we are interested in, the spacetime will satisfy not only causal continuity but also the stronger condition of {\em global hyperbolicity},\footnote{In our setup, all $J^+(p)\cap J^-(q)$ compact, see \cite{BEE,O,JS20} for background.} which implies many well-known global properties.

Indeed, in our models we can assume that $N$ is the whole $\R^n$ and, moreover, $\C$ is the cone structure determined by making $F$ equal to the usual Euclidean norm outside some compact subset $C\subset \R^n$ (i.e. $\C$ would be the cone structure of a classical Lorentz-Minkowski spacetime on $\R\times (\R^n\setminus C)$). For this spacetime, one can verify easily that all its slices $\{t=t_0\}$ become Cauchy hypersurfaces (i.e. they are crossed by any inextendible causal curve),\footnote{More precisely, it is easy to check that if a causal curve $\hat \gamma(t)=(t,\gamma(t))$, $t\in I\subset \R$, cannot be continuously extended to the endpoints of $I$, then $I=\R$.} which is a standard sufficient condition for global hyperbolicity. It is worth pointing out that any cone structure is locally globally hyperbolic.
\end{rem}

\begin{conv}\label{convention}
In what follows, we will work with the manifold $M=\R\times N$ endowed with the cone structure $\C$ determined by a cone triple $(\Omega=dt, T=\partial/\partial t,F)$, where $F$ is a Finsler metric on Ker$(dt)$ (so, identifiable with a $t$-dependent Finsler metric on $N$). Thus, $\C$ is also determined by the Lorentz-Finsler metric  $G=dt^2-F^2$ (whose lack of smoothness in the direction $\partial/\partial t$ becomes irrelevant) with fundamental tensor
$g^G=dt^2-g^F$ in \eqref{eq:rel_g_gF}.
 
For global properties, we will assume only that $(M=\R\times N, \C)$ is globally hyperbolic with Cauchy hypersurfaces $\{t=t_0\}$ (which is not restrictive for modeling). For convenience, we will assume $N\subset \R^n$, which is  restrictive neither for local computations nor for modeling.

Causal curves $\hat{\gamma}$ in $M$ will be assumed $t$-parametrized and then written $\hat{\gamma}(t)=(t,\gamma(t)), t\in I (\subset \R)$, thus future-directed. 
\end{conv}

\subsection{The case of wildfires}
\label{subsec:wildfires}
When modeling wildfire spreadings, $ N $ plays the role of a two-dimensional surface embedded in $ \R^3 $ through a graph $ \hat{z} $:
$$
\hat{z}: N \subset \R^2 \rightarrow \hat{N} \subset \R^3, \qquad 
(x,y) \mapsto \hat{z}(x,y) \coloneqq (x,y,z(x,y)),
$$
where we use the notation $ (x,y) \coloneqq (x^1,x^2) $ and $ \hat{N} \coloneqq \hat{z}(N) $ is the actual surface in $ \R^3 $ over which the fire spreads (see Fig. \ref{fig:surface}).

\begin{figure}
\centering
\includegraphics[width=1\textwidth]{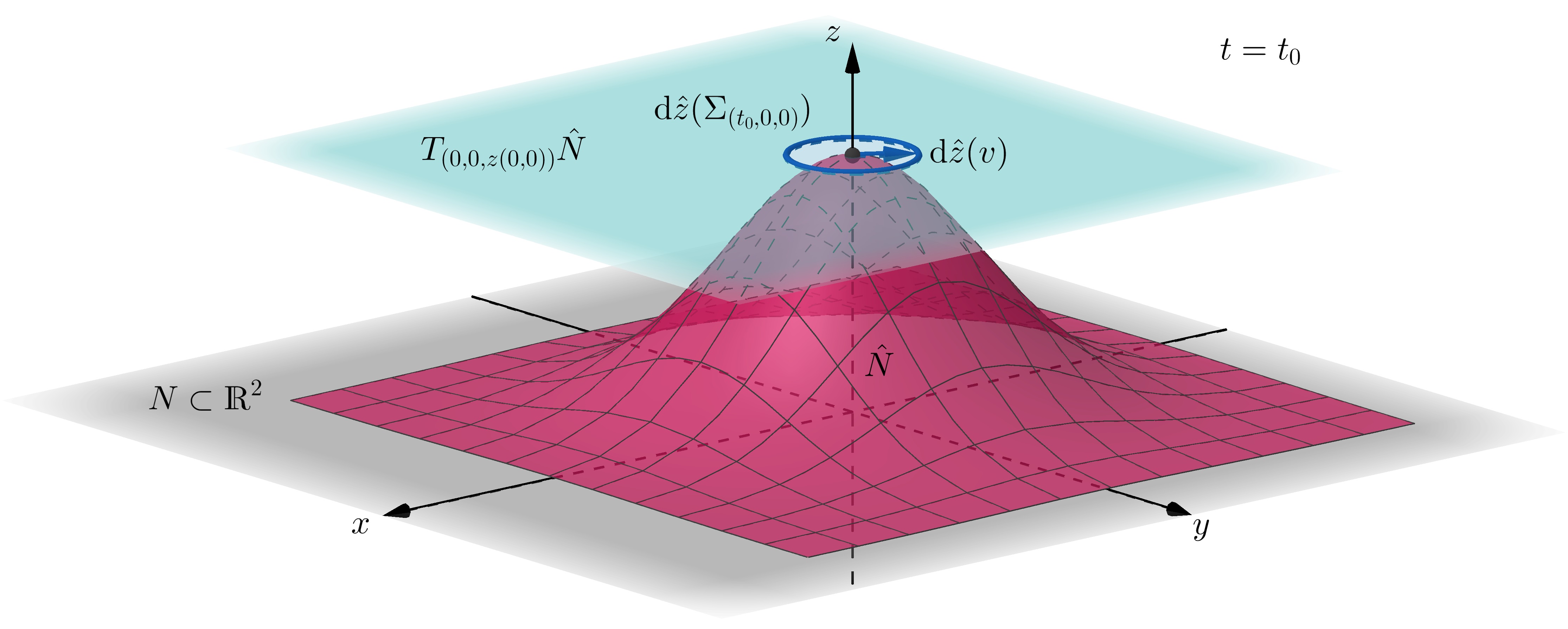}
\caption{The wildfire takes place on a surface $ \hat{N} \subset \R^3 $, with the indicatrix $ \df \hat{z}(\Sigma_p) $ on its tangent plane giving the velocity of the fire for each direction at the point $ p = (t,x,y) = (t_0,0,0) $ (implicitly, the third spatial coordinate is always assumed to be $ z(x,y) $).}
\label{fig:surface}
\end{figure}

At each point $ p = (t,x,y) \in M $, the propagation of the fire is given by the oval $ \Sigma_p $ on $ \textup{Ker}(dt_p) $. In practice, the choice of the oval at each point will depend on the fuel conditions, the wind, the slope of the surface and meteorological conditions such as the temperature, whether it is raining or not, etc. In order to model a realistic wildfire, we must allow each parameter to vary from point to point of the spacetime, i.e., they may vary in space and also in time (except for the slope, which obviously remains the same over time, inducing a Riemannian metric on $\hat N$ and thus, on $N$). The oval $ \Sigma_p $ should precisely model the infinitesimal firefront. Note that the vectors given by $ \Sigma $ are velocities on $ N $, i.e., they are the projection of the actual velocities of the firefront on $ \R^3 $, which are given by $ \df \hat{z}(\Sigma) $. Since $ \Sigma $ and $ \df \hat{z}(\Sigma) $ are equivalent (one uniquely determines the other), we can work only with $ \Sigma $ (see Fig. \ref{fig:indicatrices}).

\begin{figure}
\centering
\includegraphics[width=0.8\textwidth]{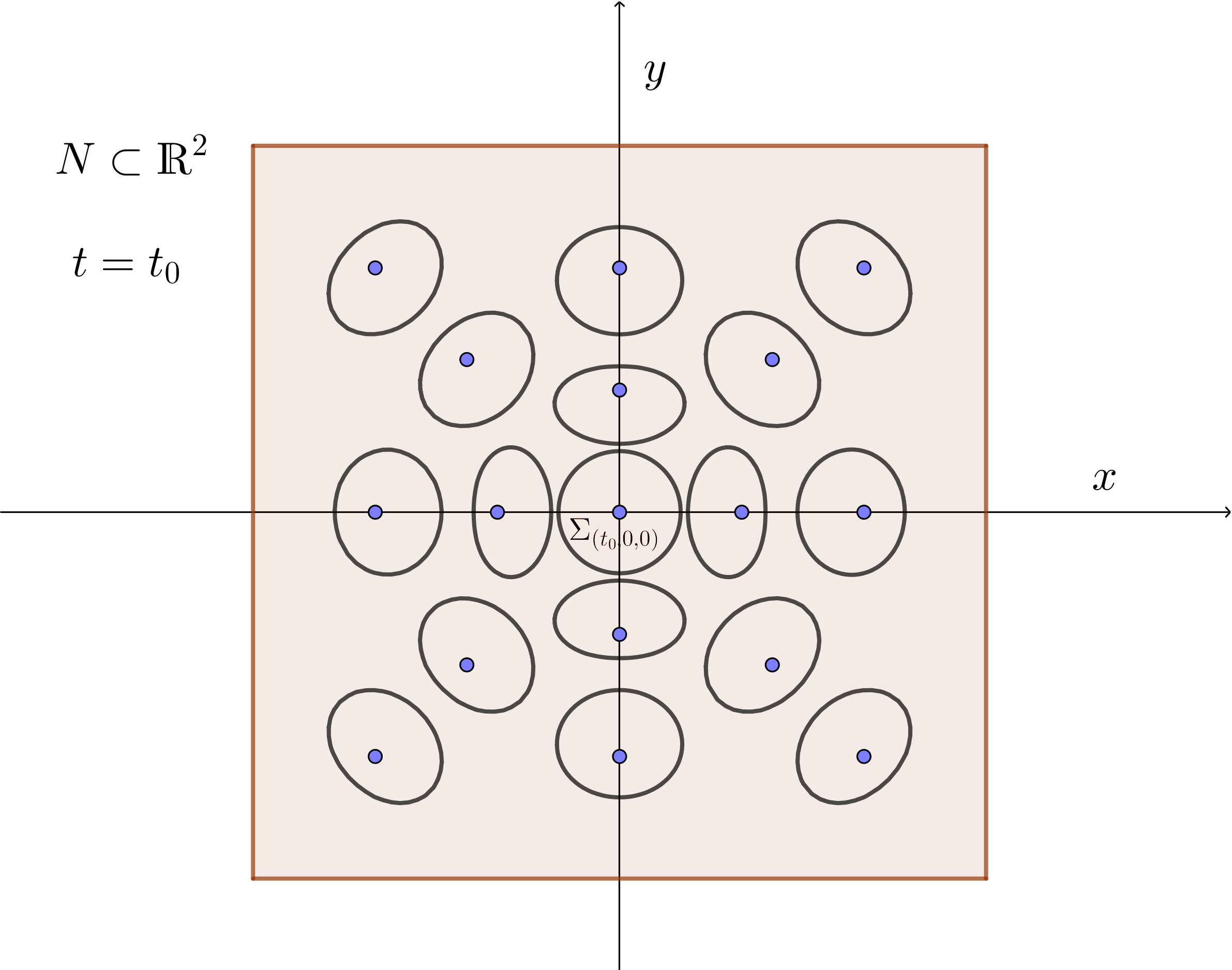}
\caption{A smooth field of strongly convex ovals  for modeling  the wildfire spread on the surface of Fig. \ref{fig:surface}. The oval $ \Sigma_p $ at each point $ p = (t_0,x,y) $ is the projection of the actual velocities on $ T_{(x,y,z)}\hat{N} $  of the firefront at $ t=t_0 $ (i.e., $ \Sigma $ contains these velocities seen from an aerial view).}
\label{fig:indicatrices}
\end{figure}

If $ B_0 = \{0\} \times B $ is the initial burned area of the wildfire, represented by a compact hypersurface of $ M $ included in $ \lbrace t=0 \rbrace $ with boundary (or simply by a unique point), then $ S_0 \coloneqq \partial B_0 $ is the initial firefront. Note that $ \partial J^+(S_0) $ provides two firefronts: the one that heads out from $ B_0 $, and the one that goes inwards. Since $ B_0 $ is already a burned area, the wavefront of interest in this case is the one pointing outwards, i.e., $ \partial J^+(B_0)$, and we have a formal model as in the case of wavefronts. Therefore, $ \overline{J^+(B_0)} \cap \lbrace t=t_0 \rbrace $ represents the total burned area at the time $ t = t_0 $ and $ \partial J^+(B_0) \cap \lbrace t = t_0 \rbrace $, the firefront at $ t = t_0 $ (see Fig. \ref{fig:cones}).

\begin{figure}
\centering
\includegraphics[width=1\textwidth]{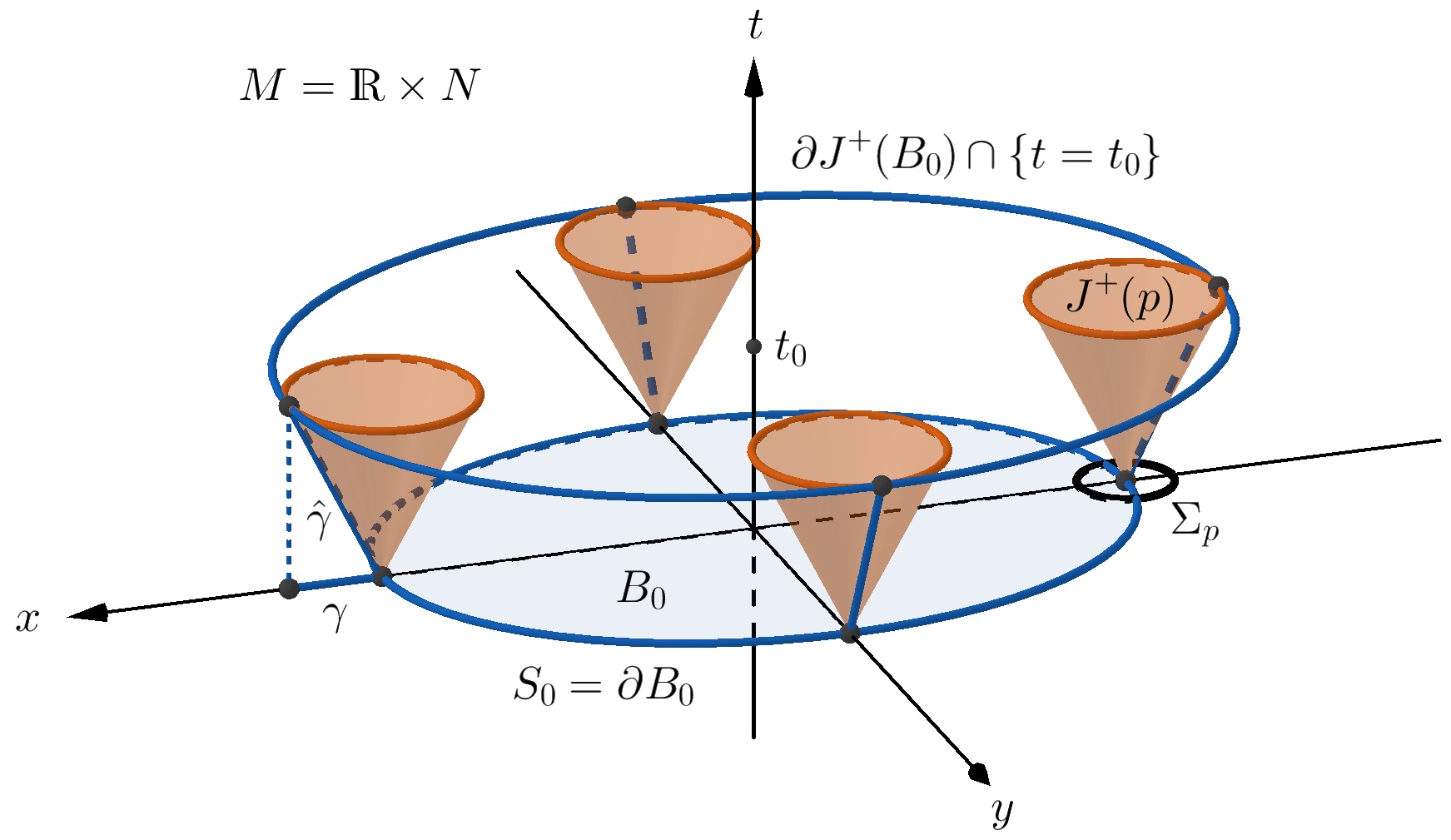}
\caption{A simple representation of wildfire evolution. The fire starts at $ t = 0 $ from the boundary $ S_0 $ of the initial burned area $ B_0 $. The causal future $ J^+(p) $ at each point $ p \in B_0 $ shows the region in the spacetime that can be reached by a point-ignited fire at $ p $. The envelope of all these chronological futures, $ \partial (\cup_{p \in B_0}J^+(p)) = \partial J^+(B_0) $, generates the outermost burned points for each instant of time, so that the intersection with $ \lbrace t=t_0 \rbrace $ is the firefront at $ t = t_0 $. The spatial  trajectory of the fire $ \gamma $ from a point at $ \partial B_0 $ is the projection of the unique causal curve $ \hat{\gamma} $ from that point entirely contained in $ \partial J^+(B_0) $; $ \hat{\gamma}$ becomes a cone geodesic and minimizes  the propagation time, at least for small $t$.}
\label{fig:cones}
\end{figure}

\section{Computation of the wavefront}\label{s4}
The initial wavefront  will be assumed to be any compact\footnote{Physically, compactness is not restrictive at all, since the wavefront must be bounded. Anyway, one can also consider here precompact manifolds with trivial modifications. As an application, this  would allow us to study compact manifolds with boundary by focusing only on their interior and, thus, avoiding the nuisance of the boundary.} embedded submanifold $ S_0= \{0\}\times S \subset M (=\R\times N)$, where $S$ has codimension $ r $ in $N$, with $1 \leq r \leq n$ (for $r= n $, $ S_0 $ is a finite number of points).

\subsection{Wavemap and minimization of the propagation time}\label{s4.1}
Let $ S_0^{\bot_G} $ (resp. $ S_0^{\bot_F} $) be the set of vectors in $ TM $ (resp. $ TN $) which are $ G $-orthogonal (resp. $ F $-orthogonal) to $ S_0 $ (recall Def. \ref{def:orth}). Note the following equivalences between working with $ G $ and $ F $: for any $ \hat{v} = (1,v) \in T_pM \setminus \text{Span}(\frac{\partial}{\partial t}\rvert_p), \hat{w} = (0,w) \in T_pM $,
\begin{equation}
\label{eq:equiv_G_F}
\begin{array}{l}
\hat{v} \text{ is lightlike} \Leftrightarrow G(\hat{v}) = 0 \Leftrightarrow F(v) = 1, \\
\hat{v} \bot_G \hat{w} \Leftrightarrow g^G_{\hat{v}}(\hat{v},\hat{w}) = 0 \Leftrightarrow g^F_{v}(v,w) = 0 \Leftrightarrow v\bot_F w
\end{array}
\end{equation}
(recall \eqref{eq:rel_g_gF} for the last one), so that
\begin{equation}
\label{eq:nu}
\begin{split}
\nu(S_0) \coloneqq & \{\hat{u}=(1,u)\in T_pM: p \in S_0, G(\hat{u})=0, \hat{u} \in S_0^{\bot_G}\} \\
= & \{\hat{u}=(1,u)\in T_pM: p \in S_0, F(u)=1, u \in S_0^{\bot_F}\}
\end{split}
\end{equation}
is the {\em normalized lightlike orthogonal bundle to} $S_0$.

\begin{lemma} $\nu(S_0)$ is a fiber bundle on $S_0$ with fiber diffeomorphic to the standard sphere $ \mathds{S}^{r-1}$.
\end{lemma}
\begin{proof} 
By the implicit function theorem, $\nu(S_0)$
(regarded as the $F$-unit vectors in $S_0^{\perp_F}$ by the second line in \eqref{eq:nu})
 is an $ (n-1) $-dimensional manifold and each $\nu(S_0)_p$ is an $ (r-1) $-dimensional submanifold of $ \textup{Ker}(dt_p) $ (see, e.g., \cite[Lem. 3.3]{JavSoa15}). So, it is enough to construct a diffeomorphism from 
 $\nu(S_0)_p$ to $ \mathds{S}^{r-1}$ smoothly depending on $p$. 
The former is the set containing all $u\in \Sigma_p$ such that the affine subspace through $u+T_pS_0$ is tangent to $\Sigma_p$. 
Moreover, all these subspaces $u+T_pS_0, u\in (S_0^{\perp_F})_p$ become a cylinder with base $(S_0^{\perp_F})_p$ whose affine second fundamental form is positive semi-definite with 
radical identifiable to $u+T_pS_0$, and each  $u+T_pS_0$ 
intersects $\Sigma_p$ only at $u$ (because $\Sigma_p$ is strictly convex). So, taking any $r$-linear complementary subspace  $(T_pS_0)^c$ (such that 
$\text{Ker}(dt_p)=T_pS_0 \oplus (T_pS_0)^c$), the map 
$$
\nu(S_0)_p \ni u \rightarrow (u+T_pS_0)\cap (T_pS_0)^c
\in (T_pS_0)^c$$ 
is an embedding of $\nu(S_0)_p$ in $(T_pS_0)^c$ as 
 a compact, embedded, strongly convex hypersurface and, then, a diffeomorphism onto $ \mathds{S}^{r-1} $.
\end{proof}

\begin{conv}
\label{convention2}
Locally, $\nu(S_0)$ is diffeomorphic to $S\times \mathds{S}^{r-1}$ and, when working in coordinates, we can assume that this property holds globally, i.e., $\nu(S_0) \cong S \times \mathds{S}^{r-1}$.\footnote{It will be satisfied automatically in the case of wildfires ($S_0=\partial B_0$, $n=2$, $r=1$), as $\mathds{S}^{0}$ only contains two points, and the one pointing outwards from $B_0$ is selected.} So, any $ \hat u\equiv (1,u)\in \nu(S_0)$ at $ p=(0,s) \in S_0 $ can be identified with $(s,u)\in S\times \mathds{S}^{r-1}$, where $u\in \mathds{S}^{r-1} \cong \Sigma_{(0,s)}\cap S_0^{\perp_F} \subset $ Ker$(dt_{(0,s)}) \equiv T_sN$, consistently with Notation \ref{notacion}.
\end{conv}

Using the convention above, the {\em wavemap}
\begin{equation} \label{def:wave_map}
\hat{f}: [0,\infty) \times  \nu(S_0)  \rightarrow M , \qquad  
(t,\hat u) \equiv (t,s,u) \mapsto \hat{f}(t,s,u) 
\end{equation}
is defined so that for each $\hat u=(1,u)\in \nu(S_0)$ ($u$ tangent to $ s \in S$), the curve $ t \mapsto \hat{f}(t,s,u) $ is the unique cone geodesic $t$-parametrized  with initial velocity $ \hat u$. Clearly, this function is smooth for $t>0$ (as so is the exponential map of $G$ on lightlike directions) and continuous at $t=0$. From Thm. \ref{cor:zermelo} below,  $\hat f([0,\varepsilon),\nu(S_0))$ becomes an embedded topological hypersurface of $M$ which is smooth up to $S_0$, for small $\varepsilon>0$.

As $ \partial J^+(S_0) $ generates the wavefront for each instant of time (recall \eqref{e_wavefront_hatf} with $S_0=$front$(0)$), the causal curves starting at $S_0$ contained in $ \partial J^+(S_0) $ can be regarded as the (outermost) spacetime trajectories of the wave. These curves will represent first-arriving perturbations in the following sense. 

\begin{defi}
Let $ \hat{\gamma}(t)=(t,\gamma(t)), t\in I$, be a  causal curve departing from $S_0$. We say that $\hat\gamma$ is {\em first-arriving} (resp. {\em strictly first-arriving}) if, for each $t_0\in I$, $x_0=\gamma(t_0)$, any other causal curve $\hat{\alpha}$ departing from $S_0$ with $\alpha(t_1)=x_0$ satisfies $t_1\geq t_0$ (resp. $t_1 > t_0$).\footnote{This is equivalent to saying that $\hat{\gamma}$ is a solution of Zermelo's navigation problem, see \cite{CJS,JS20}.} 

In this case, $\gamma$ is a {\em spatial trajectory} of the wave. 
\end{defi}

However, let us see that the unique causal curves in $ \partial J^+(S_0) $ (and, thus, the unique  first-arriving curves) will be its so-called null generators. Indeed, $ \partial J^+(S_0)$ is an {\em achronal boundary},\footnote{Notice that, for any  $ S_0 \subset M $, $ \overline{J^+(S_0)} = \overline{I^+(S_0)} $ and, thus, $ \partial J^+(S_0) = \partial I^+(S_0)$.} where achronal means that no pair of its points can be connected by a timelike curve entirely contained in  $ \partial J^+(S_0)$. The theory of these boundaries is well-established in the Lorentzian setting \cite{G} and we will use in the next proposition only some properties which can be directly transplanted to the Lorentz-Finsler setting \cite{AJ,JS20,Minguzzi} (anyway, detailed computations will be available in \cite{tesis}).

\begin{prop}
\label{prop:zermelo}
$ \partial J^+(S_0)\setminus S_0$ is a locally Lipschitz hypersurface and it admits a unique foliation by lightlike geodesics (null generators) of $G$, i.e., cone geodesics of $\C$. Moreover, because of the global hyperbolicity of $\C$ (Conv. \ref{convention}), such a geodesic $ \hat \gamma $ must always reach $ S_0 $ once and, at that point, $ \hat\gamma $ must be $ G $-orthogonal to $ S_0 $.
\end{prop}
\begin{proof}
The first sentence is standard for any achronal boundary,\footnote{In general, one should add ``if $ \partial J^+(S_0)\setminus S_0$ is not empty'', but this holds trivially in our case.} and the  notion of null generators is well known (see, e.g., \cite{G} for the Lorentzian case and \cite{tesis} for its translation to Finsler spacetimes). For the last one, any $\hat\gamma(t)$ belongs to $J^+(S_0)$ and, then, $J^-(\hat\gamma(t))\cap J^+(S_0)$ is compact and $\hat\gamma$ must have an initial point $p_0=\hat\gamma(t_0)$. However, if $t_0>0$ then a causal curve from $S_0$ to $\gamma(t_0)$ would exist. Concatenating it with $\hat\gamma$ one finds a causal curve from $S_0$ to $\hat\gamma(t)$ which is not a lightlike geodesic and, thus, $\hat\gamma(t)\not\in \partial J^+(S_0)$. To check orthogonality, observe that otherwise, for any $t>0$, $\hat\gamma(t)$ would lie in $J^+(p_0')$ for some $p_0'\in S_0$ close to $p_0$ (see \cite[Prop. 6.4]{AJ}).
\end{proof}

As a consequence of this proposition, $\partial J^+(S_0)$ must lie in the image of the wavemap \eqref{def:wave_map}. More precisely, if
$$
c: \nu(S_0)\rightarrow [0,\infty], \qquad (s,u) \mapsto c(s,u)=\hbox{Max} \{t: \hat f(t,s,u)\in \partial J^+(S_0)\}
$$
is the {\em null cut function from} $S_0$, then  
$$\partial J^+(S_0)=\{\hat f(t,s,u): \; t\leq c(s,u),\, (s,u)\in \nu(S_0) \}$$
and, for any $t_1\geq 0$, front$(t_1)$ is obtained just considering the points with $t=t_1$ in $ \partial J^+(S_0) $. 

\begin{rem}
For a complete Riemannian manifold, it is well known that a geodesic emanating from a compact submanifold $S_0$ strictly minimizes the distance before its cut instant $t_c$ (which appears not later than the first focal point) and there will be shorter geodesics from $S_0$ after $t_c$ (see, e.g., \cite[Prop. 2.2]{dC} and \cite[Lem. 2.11]{Sa}). Such properties have a direct translation for the cut points of lightlike geodesics for any globally hyperbolic Lorentz \cite[\S 9]{BEE} or Lorentz-Finsler metric such as our $G$ \cite{tesis}. 
\end{rem}

So, the following essential result follows as a straightforward consequence of Prop. \ref{prop:zermelo} and the definition of the null cut locus. 

\begin{cor}
\label{cor:zermelo0}
The only first-arriving causal curves from $S_0$ are the cone geodesics of $\C$ departing orthogonally from $S_0$ until they arrive at their cut points. Thus, they lie in the image of the wavemap. 
\end{cor}

However, the behavior of $c$ is subtle even in the Lorentz case. Indeed, for a compact submanifold $S$ of a complete Riemannian manifold $(N,g_R)$, the cut locus is known to be continuous (and even locally Lipschitz where finite \cite{IT}). Such property is transmitted directly for the lightlike geodesics of the Lorentzian metric $g_L=dt^2-g_R$ on $\R\times N$, which is globally hyperbolic (noticeably, see \cite[\S 4]{CFGH}). Nevertheless, in general the null cut function of a point in globally hyperbolic spacetimes is known to be only lower-semicontinuous \cite[Prop. 9.33]{BEE}. As we will be interested only in the property $c \geq \varepsilon$ for some $\varepsilon>0$, a self-contained proof is provided next. It is worth pointing out that this result, applied to $G=dt^2-F^2$ with $F$ a Finsler metric, yields tubular neighborhoods for submanifolds of a Finsler manifold (see \cite{AlJ} for a direct proof).

\begin{lemma}
\label{th:zermelo}
Given $\hat  v \in \nu(S_0) $, there exists $ \varepsilon > 0 $ and a neighborhood $ \hat{W} $ of $ \hat v $ in $ \nu(S_0) $ such that all the cone geodesics $ \hat{\gamma} $ with initial velocity $ \hat{\gamma}'(0) \in \hat{W} $ are strictly first-arriving from $ S_0 $ in the interval $ [0,\varepsilon] $.
\end{lemma}
\begin{proof}
Given $ \hat v = (1,v) \in \nu(S_0) $, consider a chart $ (U,\varphi) $ around $ \pi^N(v) \in N $ adapted to $ S $. Without loss of generality, we can assume that $ \varphi(U)=\R^n $, so that $ \varphi(S\cap U) $ is a linear subspace of $ \R^n $. Consider now an open neighborhood $ \hat W \subset \nu(S_0)$ of $ \hat v $, with $W \coloneqq \pi^{\Omega}(\hat W)\subset TN$ included in the coordinate neighborhood of $ TN $ induced naturally from $ (U,\varphi) $. For each $ \hat{w} = (1,w) \in \hat{W} $, $ s = \pi^N(w) $, $ w $ is $ F_{(0,s)} $-orthogonal to $ S $ with $ F(w) = 1 $ (recall \eqref{eq:nu}), and we can choose a basis $\lbrace e_1(w),\ldots,e_{r-1}(w),w \rbrace $ of the $ g^{F_{(0,s)}}_w $-orthogonal space to $ S $ ($ g^{F_{(0,s)}} $ is the fundamental tensor of the Minkowski norm $ F_{(0,s)}: T_sN \rightarrow \R$). Working in our coordinates on $TN$ (adapted to $ S $) and applying Gram-Schmidt, this basis can be chosen $ g^F_w $-orthogonal and with a smooth dependence on $ w $.

Although the searched property of being first-arriving is global on $S_0$, we can work locally. Indeed, recall that, for any precompact $\tilde N\subset N$, $ [0,1] \times \tilde N$ admits a flat Lorentz-Minkowski metric with wider cones than $ \mathcal{C}$ (this is consistent with Conv. \ref{convention}). Now consider
\begin{equation}
\label{eq:U1U2}
U_1, U_2 \text{ neighborhoods of } v\coloneqq\pi^{\Omega}(\hat v), \text{ with } \overline{U_1} \subset U_2, \overline{U_2} \subset \tilde{N} \subset U.
\end{equation}
The Lorentz-Minkowski causal future of $\{0\}\times \partial U_2$ does not intersect $ [0,\varepsilon] \times U_1 $ for some small enough $\varepsilon$ and, thus, neither does the $ \mathcal{C} $-causal future of $ \lbrace 0 \rbrace \times (N \setminus U_2)$. This means that we only need to prove the searched property on a suitable $ \{0\} \times (S \cap U_2) $.

Moreover, as $ S \subset N $ is, in general, a submanifold of arbitrary codimension $ r $, we will reduce the proof to the case of $ r = 1 $ by constructing, for each $ w \in W $, a hypersurface $ S^w $ that contains $ S \cap U $ and such that $ w $ is still $ F $-orthogonal to $ S^w $.\footnote{Apart from other reasons pointed out above, we proceed this way because the space of $ F $-orthogonal vectors to $ S $ is not a vector bundle in general, but a submanifold with a conical singularity in the zero section, except when $ S $ is a hypersurface, in which case we have two one-dimensional vector bundles, one in each face.} This way, it will suffice to prove that every cone geodesic (at least in a small enough interval $ [0,\varepsilon] $ independent of $ w $) with initial velocity in a sufficiently small $ W' \subset W $ is strictly first-arriving from $ \{0\} \times (S^w \cap U_2) $, for all $ w \in W' $.

For each $ w \in W $, regard $ P^w \coloneqq \text{Span}(\{e_1(w),\ldots,e_{r-1}(w)\}) $ as a linear subspace of $ \R^n $ (using the coordinates in $ T_sN $) and let
$$
S^w \coloneqq \cup_{y \in (S\cap U)}\varphi^{-1}(\varphi(y)+P^w), 
$$
obtained by adding the coordinates of each $y\in S\cap U$ and those in $P^w$. Observe that $ \varphi(S^w) $ is a hyperplane of $ \R^n $, so that $ S^w $ trivially becomes a hypersurface of $ N $ that contains $ S \cap U $ in such a way that $ w $ is $ F_{(0,\pi^N(w))} $-orthogonal to $ S^w $, as required. Note also that due to the orientability of $ S^w $, the set of $ F $-orthogonal vectors to $ S^w $ has two connected components.

We now proceed to obtain a map $ \eta $ that will play the role of a ``smooth exponential map'' (the true exponential map fails to be smooth at 0). To this end, first define the smooth map $ \hbox{nor}: W \times S \cap U \times \R^n \rightarrow TN $, where $\hbox{nor}(w,y,\ell) $ is the unique $ F $-unit vector $ F $-orthogonal to $ S^w $ at the point $ \varphi^{-1}(\varphi(y)+\ell_1e_1(w)+\ldots+\ell_{r-1}e_{r-1}(w)) \in S^w $ and in the same connected component as $ w $ (so that $ \hbox{nor}(w,\pi^N(w),0) = w $). A dual mapping $\hbox{nor}^-$ obtained by choosing the normal vector in the other connected component will be used too.
 
Now, define $ \eta: A \rightarrow W \times N $, where $ A $ is an open subset of $ \R \times W \times S \cap U \times \R^n $, as
$$
\eta(t,w,y,\ell) \coloneqq (w,\pi_N(\gamma^G_{\hbox{\tiny{nor}}(w,y,\ell)}(t))),
$$
being $ \gamma^G_{\hbox{\tiny{nor}}(w, y,\ell)} $ the t-parametrized geodesic in $ (M,G) $ with initial velocity $ (1,\hbox{nor}(w, y,\ell))\in \hat W $ at $ t = 0 $. Observe that for each point with $ t = 0 $, $ \df \eta $ is an isomorphism. Therefore, there exists a restriction of $ \eta $ in a neighborhood of $ (0,v,\pi^N(v),0) $ where it is a diffeomorphism and its image is of the form $W\times \tilde N$, with $ \tilde N $ precompact and $ \varphi(\tilde N) $ convex (as a subset of $ \R^n $).

Let $W^-\times \tilde N^-$ be the analogous image one would obtain for the mapping $ \eta^- $ (constructed using $ \text{nor}^- $). Choose $ U_1,U_2 $ as in \eqref{eq:U1U2} with the additional condition $ \overline{U_2} \subset \tilde N^- $ and the resulting $ \varepsilon > 0 $, and take a neighborhood $ W' \subset W \cap W^- $ of $ v $ small enough to ensure that $\gamma^G_{\hbox{\tiny{nor}}(w,\pi^N(w),0)}(t) $ and $ \gamma^G_{\hbox{\tiny{nor}}^-(w,\pi^N(w),0)}(t) $ are included in $ [0,\varepsilon]\times U_1 $ for all $ t \in [0,\varepsilon], w\in W' $ (reducing $ \varepsilon $ if necessary). Observe that this choice of $ \varepsilon $ guarantees that the normal geodesics to $ S^w $ starting on different sides do not intersect. Indeed, note that $ S^w $ divides $ \tilde N \cup \tilde N^- $ into two connected components and the projection $ \pi_N(\gamma^G_{\hbox{\tiny{nor}}(w,\pi^N(w),0)}(t)) $, with $ t\in [0,\varepsilon], w \in W' $, must remain entirely in one of them. Otherwise, it would have to cross $ S^w $ in order to pass to the other component (as it cannot escape $ \tilde N \cup \tilde N^- $ in the chosen interval), but this yields a contradiction with the fact that $ \eta $ is a diffeomorphism. The same happens with the normal geodesics associated with $ \eta^- $, which must remain on the opposite connected component.

Observe that the construction of $ U_1 $ and $ U_2 $ ensures that no causal curve departing from $ N \setminus U_2 $ enters $ [0,\varepsilon] \times U_1 $, so we only need to prove that each $t$-parametrized cone geodesic $ \hat{\gamma}_w(t)=(t,\gamma(t)) $, $\hat{\gamma}'(0) = (1,w) $,  $t\in [0,\varepsilon]$ (note that $ \hat \gamma_w(t)=\gamma^G_{\hbox{\tiny{nor}}(w,\pi^N(w),0)}(t) \in [0,\varepsilon] \times U_1 $), is first-arriving from $ \lbrace 0 \rbrace \times (S^w \cap U_2) $, for all $w \in W'$.
Otherwise, there exists $ \mu < \varepsilon $ such that $ (\mu,\gamma_w(\varepsilon))\in J^+(\lbrace 0 \rbrace \times (S^w\cap U_2))$. Put
$$
\mu_0\coloneqq\inf\{\mu>0:(\mu,\gamma_w(\varepsilon)) \in J^+(\lbrace 0 \rbrace \times (S^w \cap U_2)) \} >0
$$
(the inequality holds for fixed $w$). Without loss of generality, we can assume that the closure of $ [0,\varepsilon] \times (\tilde{N} \cap \tilde{N}^-) $ is contained in a convex neighborhood of $ (M,G) $,\footnote{Such a neighborhood is a normal neighborhood of all its points, so that the exponential map at each point will be a diffeomorphism up to the origin (because of its Finslerian character). In the case of $(M,G)$, the non-smooth direction $\partial/\partial t$ would remain non-smooth for the exponential too. However, this will not be relevant for our case because, as noted in Rem. \ref{rem:lf_metric}, $G$ can be smoothen along $G$ preserving the metric around the cone structure $\C$; moreover, only properties of cone geodesics (as those in \cite{AJ}) will be claimed. The existence of convex neighborhoods was proved by Whitehead (first for linear connections \cite{W1} and then extended to sprays \cite{W2}).} so $ (\mu_0,\gamma_w(\varepsilon)) \in J^+(\lbrace 0 \rbrace \times (S^w\cap U_2))$. As $ (\mu_0,\gamma_w(\varepsilon)) \not\in I^+(\lbrace 0 \rbrace \times (S^w\cap U_2))$ by the definition of $\mu_0$, this point is reached by a lightlike geodesic from $ \lbrace 0 \rbrace \times (S^w\cap U_2) $ which is $G$-orthogonal to $ \lbrace 0 \rbrace \times S^w $ (see \cite[Thm. 6.9]{AJ}), so that it turns out that $ \eta $ is not injective in $ W \times \tilde N $, which is a contradiction. Finally, $\hat\gamma_w$ is also strictly first-arriving because, otherwise, another causal curve $ \hat\alpha $ from $\{0\}\times S^w\cap U_2$ would arrive at $\hat\gamma_w(\varepsilon)$, which is on the boundary of the causal future $J^+(\lbrace 0 \rbrace \times (S^w\cap U_2))$ (as $ \hat\gamma $ is first-arriving on $[0,\varepsilon]$). Therefore, $ \hat\alpha $ is necessarily an orthogonal lightlike geodesic, in contradiction with $\eta$ being injective, which concludes.
\end{proof}

So, the following result becomes trivial from the compactness of $\nu(S_0)$.

\begin{thm}
\label{cor:zermelo}
The null cut function from $S_0$ satisfies $c>\varepsilon$ for some $\varepsilon>0$, that is, for small time, every cone geodesic with initial velocity in $ \nu(S_0) $ is strictly first-arriving from $S_0$.
\end{thm}

\begin{rem}\label{r_resumen_s4.2}
Summing up, the curves that minimize the propagation time from $ S_0 $ are the cone geodesics $ G $-orthogonal to $ S_0 $, which are also the only causal curves contained in $ \partial J^+(S_0) $. They remain time-minimizing at least in a short common lapse. However, each geodesic $\hat \gamma(t)$ will leave $ \partial J^+(S_0) $ if it has a cut point and, immediately after this point, global hyperbolicity implies that a second cone geodesic from $S_0$ will reach it first. The projection $\gamma$ is the spatial trajectory of the wave (i.e., the wave propagates faster along $\gamma$), but different first-arriving trajectories will meet beyond the cut point.
\end{rem}

\subsection{Cone geodesics and spatial trajectories of the wave}
\label{subsec:cone_geod} Next, the wavemap $\hat f$ will be determined by obtaining first the geodesic equations of $G$ and, then, the equation of the reparametrization required for $\hat f$. As we will work in coordinates, we will use $\nu(S_0)=S\times \mathds{S}^{r-1}$ (Conv. \ref{convention2}).

Put $ \hat{f}(t,s,u) = (t,f(t,s,u)) = (t,x^1(t,s,u),\ldots,x^n(t,s,u)) $ and consider the notation $ \partial_t\hat{f}(t,s,u) \coloneqq \frac{\partial \hat{f}(t,s,u)}{\partial t} = \df \hat{f}_{(t,s,u)}\left(\frac{\partial}{\partial t}\right)$, that is,\footnote{In Finslerian notation, tangent vectors are usually written using the coordinates of $ TM $, i.e., one writes the coordinates of the point $ p \in M $ and the coordinates of the vector $ v \in T_pM $. Here we will omit the point coordinates, as there is no ambiguity regarding its identification given the vector.}
\begin{equation}\label{not:vel_t}
\partial_t\hat{f}(t,s,u) = (1,\partial_tf(t,s,u)) = (1,\partial_tx^1(t,s,u),\ldots,\partial_tx^n(t,s,u)).
\end{equation}
Also, $ \hat{f}_{t_0}(s,u) = (t_0,f_{t_0}(s,u)) = (t_0,x_{t_0}^1(s,u),\ldots,x_{t_0}^n(s,u)) $ will denote the wavemap at time $ t_0 \geq 0 $. 
So, fixing $ s_0 \in S $ and $ u_0 \in \mathds{S}^{r-1} $ (recall Conv. \ref{convention2}), the curve $ t \mapsto \hat{f}(t,s_0,u_0) $ provides the spacetime trajectory of the wave from $ (0,s_0)\in S_0 $ in the direction $\hat u \equiv (s_0,u_0) \in \nu(S_0)$, and $ t \mapsto f(t,s_0,u_0) $ provides the spatial trajectory; fixing $ t_0 \geq 0 $, $ \hat{f}_{t_0} $ generates the wavefront at the instant $ t = t_0 $.

In general, $G$-geodesics are not parametrized by $t$. So, let us introduce an arbitrary (smooth) reparametrization in time $\tilde{t}(\rho,s,u)$ with $\partial_{\rho} \tilde{t}(\rho,s,u)>0$ ($ \rho $ will be the parameter of the geodesic). Namely, let 
$ \tilde{f} \coloneqq \hat{f} \circ \psi = (\tilde{t},\tilde{x}^1,\ldots,\tilde{x}^n) $, where
$$
\begin{array}{cccccc}
  \tilde f \colon & [a,\infty) \times S \times \mathds{S}^{r-1} & 
  \overset{\psi} \longrightarrow 
  & [0,\infty) \times S \times \mathds{S}^{r-1} & 
\overset{\hat f} \longrightarrow  
  & M\\
 & (\rho,s,u) & \longmapsto & (\tilde t(\rho,s,u),s,u) & \longmapsto & \hat{f}(\tilde t(\rho,s,u),s,u),
\end{array}
$$
so that
$ \tilde{f}(\rho,s,u) = (\tilde{t}(\rho,s,u),\tilde{x}^1(\rho,s,u),\ldots,\tilde{x}^n(\rho,s,u))$ with $\tilde{x}^i(\rho,s,u)=x^i(\tilde t(\rho,s,u),s,u)$.

\begin{notation}
\label{not:einstein}
In order to simplify summations and avoid clutter, we will use the notation $ x^0 \coloneqq t $ and analogously, $ \tilde{x}^0(\rho,s,u) \coloneqq \tilde{t}(\rho,s,u)$, when convenient. Consistently, $(x^0,x^1,\ldots,x^n) $ will also denote the natural coordinate functions on $ M=\R\times N$, so that $ g_{ij}(\hat v) \coloneqq g^G_{\hat v}(\frac{\partial}{\partial x^i},\frac{\partial}{\partial x^j}) $ for any $ \hat v \in TM \setminus \textup{Span}(\frac{\partial}{\partial x^0}) $ and $ g^{ij}(\hat v) $ will be the coefficients of the inverse matrix of $ \{g_{ij}(\hat v)\}$. Moreover, Einstein's summation convention will be used, i.e., we will omit the sums from $ 0 $ to $ n $ when an index appears up and down, and we will raise and lower indices using $ g_{ij} $ and $ g^{ij} $.
\end{notation}

The \textit{Christoffel symbols} $ \Gamma_{\ ij}^k(\hat v) $ of $ (M,G) $ in the direction $ \hat v \in TM \setminus \textup{Span}(\frac{\partial}{\partial x^0}) $ are given by 
$$
\nabla^{\hat v}_{\frac{\partial}{\partial x^i}}\left(\frac{\partial}{\partial x^j}\right) = \Gamma_{\ ij}^k(\hat v)\frac{\partial}{\partial x^k}, \quad i,j = 0,\ldots,n,
$$
where $ \nabla $ is the Chern connection. Also, the \textit{formal Christoffel symbols} $ \gamma_{\ ij}^k(\hat v) $ (see \cite[\S 2.3]{BCS}) are defined as
\begin{equation}
\label{eq:christ_formal}
\gamma_{\ ij}^k(\hat v) \coloneqq \frac{1}{2}g^{kr}(\hat v)\left(\frac{\partial g_{rj}}{\partial x^i}(\hat v)+\frac{\partial g_{ri}}{\partial x^j}(\hat v)-\frac{\partial g_{ij}}{\partial x^r}(\hat v)\right).
\end{equation}
Note that the dependence with $ \hat v $ appears only in the (pointwise) direction, i.e.,
\begin{equation}
\label{e_direction} 
\Gamma_{\ ij}^k(\lambda \hat v) = \Gamma_{\ ij}^k(\hat v), \qquad  \gamma_{\ ij}^k(\lambda \hat v) = \gamma_{\ ij}^k(\hat v)
\end{equation}
for any  function $ \lambda>0$. This property  is characteristic of Finsler geometry, as $\nabla^{\hat v} $ and $ g^G_{\hat v} $ become positively zero-homogeneous. 
 
Fixing $ s \in S $ and $ u \in \mathds{S}^{r-1} $, the geodesic equation for the curve $ \tilde{f}(\rho,s,u) = (\tilde{x}^0(\rho,s,u),\ldots,\tilde{x}^n(\rho,s,u)) $ is
$$
\partial_{\rho}^2\tilde{x}^k = -\Gamma_{\ ij}^k(\partial_{\rho}\tilde{f})\partial_{\rho}\tilde{x}^i\partial_{\rho}\tilde{x}^j, \quad k = 0,\ldots,n,
$$
but only the formal Christoffel symbols contribute to the double contraction on the right (see \cite[\S 5.3]{BCS}) and it becomes
\begin{equation}
\label{eq:geod_tau_gen}
\partial_{\rho}^2\tilde{x}^k = -\gamma_{\ ij}^k(\partial_{\rho}\tilde{f}) \; \partial_{\rho}\tilde{x}^i\partial_{\rho}\tilde{x}^j, \quad k =  0 \ldots,n,
\end{equation}
where
\begin{equation}\label{e_18}
\begin{split}
\gamma_{\ ij}^k(\partial_{\rho}\tilde{f}) = & \gamma_{\ ij}^k(\tilde{x}^0,\ldots,\tilde{x}^n,\partial_{\rho}\tilde{x}^0,\ldots,\partial_{\rho}\tilde{x}^n) \\
= & \gamma_{\ ij}^k \left( \tilde{x}^0,\ldots,\tilde{x}^n,1,\frac{\partial_{\rho}\tilde{x}^1}{\partial_{\rho}\tilde{t}},\ldots,\frac{\partial_{\rho}\tilde{x}^n}{\partial_{\rho}\tilde{t}} \right),
\end{split}
\end{equation}
the latter using \eqref{e_direction} (recall $\partial_{\rho}\tilde{x}^0 \equiv \partial_{\rho} \tilde t$).

\begin{thm}
\label{th:geod_gen}
For each $ (s,u) \in S\times \mathds{S}^{r-1} (\cong \nu(S_0))$, the wavemap $ \hat{f}(t,s,u) = (t,f(t,s,u)) = (x^0,x^1(t,s,u)\ldots,x^n(t,s,u)) $ is given by the following ODE system:
\begin{equation}
\label{eq:geod_t_gen}
\partial_t^2x^k = -\gamma_{\ ij}^k(\partial_t\hat{f})\partial_tx^i\partial_tx^j + \gamma_{\ ij}^0(\partial_t\hat{f})\partial_tx^i\partial_tx^j\partial_tx^k, \quad k=1,\ldots,n.
\end{equation}
Therefore, the spatial trajectories of the wave are the solutions $ f(t,s,u) = (x^1(t,s,u),\ldots,x^n(t,s,u)) $ whose initial conditions satisfy
\begin{itemize}
\item $ f(0,s,u) = (x_0^1(s,u),\ldots,x_0^n(s,u)) = s $,
\item $ \partial_tf(0,s,u) = (\partial_tx^1(0,s,u),\ldots,\partial_tx^n(0,s,u)) = u \ (\in \Sigma_{(0,s)}\cap S_0^{\perp_F}\cong \mathds{S}^{r-1}).$  
\end{itemize}
\end{thm}
\begin{proof}
By the definition of the wavemap, the curve $ t \mapsto \hat{f}(t,s,u) $ must be a lightlike pregeodesic parametrized by the time $ t $ and with the initial conditions enunciated above. The reparametrization that makes it geodesic is precisely $ \tilde{f} = \hat{f} \circ \psi $, with $ \tilde{f}(\rho,s,u) = (\tilde{x}^0(\rho,s,u),\ldots,\tilde{x}^n(\rho,s,u)) $ satisfying \eqref{eq:geod_tau_gen}. Therefore, we need to rewrite this geodesic equation in terms of the parameter $ t $. As
$ \partial_{\rho}\tilde{x}^i =  \partial_{\rho}\tilde{t} \, \partial_tx^i $,
\begin{equation}
\label{eq:chain_rule}
\partial_t^2x^k = \frac{1}{\partial_{\rho}\tilde{t}} \partial_{\rho} \left(\frac{\partial_{\rho}\tilde{x}^k}{\partial_{\rho}\tilde{t}}\right) 
= \frac{1}{(\partial_{\rho}\tilde{t})^2} \left(\partial_{\rho}^2\tilde{x}^k- \partial_{\rho}^2\tilde{t} \, \frac{\partial_{\rho}\tilde{x}^k}{\partial_{\rho}\tilde{t}}\right), 
\quad k = 1,\ldots,n,
\end{equation}
and \eqref{eq:geod_t_gen} follows substituting \eqref{eq:geod_tau_gen} (for the chosen $k$ and $k=0$) in \eqref{eq:chain_rule} taking into account \eqref{e_18}.
\end{proof}

Note that the right-hand side of \eqref{eq:geod_t_gen} vanishes for $k=0$, consistently with the $t$-reparametrization of the trajectories.

\subsection{Trajectories for wildfires and PDE's}
\label{subsec:orth_cond}
We turn our attention to the particular case of dimension $ n = 2 $ and $ r = 1 $, with $ S_0 $ being the boundary of a compact hypersurface $ B_0 \subset \lbrace t=0 \rbrace $ of $ M $ with boundary, which is the situation when modeling wildfires, \S \ref{subsec:wildfires}. In this case, at each $ p \in S_0 $, $ \nu(S_0)_p $ is homeomorphic to $ \mathds{S}^0 $ (it contains  two points) and we will be interested only in the one representing the lightlike direction whose projection on $TN$ points outwards from $ B_0 $. This way, the dependence on $ u \in \mathds{S}^0 $ is dropped and the wavemap becomes a function $ \hat{f}: [0,\infty) \times S \rightarrow M $. Observe that, now, the image of $ \hat{f} $ includes $ \partial J^+(B_0) \setminus \textup{Int}(B_0) $ (and it is equal to this set for small $t$). 

\begin{conv}\label{conventionfire}
$ S_0 $ is assumed to be connected (otherwise, each connected component would be taken into account separately) and, thus, diffeomorphic to $\mathds{S}^1 $. Implicitly, the wavemap provides a parametrization of $S_0$, $[a,b]\ni s\mapsto \hat f(0,s) \in S_0$, with $\hat f(0,a)=\hat f(0,b)$, but we do not have any preferred parametrization. As in \eqref{not:vel_t}, $ \partial_s\hat{f}(t,s) $ will denote the velocity of the curve $ s \mapsto \hat{f}(t,s) $, 
$$
\partial_s\hat{f}(t,s) = (0,\partial_sf(t,s)) = (0,\partial_sx^1(t,s), \partial_sx^2(t,s)),
$$
and $ \partial_s\hat{f}(t_0,s) $ is a basis of the tangent space of front$(t_0)$ whenever $t_0<c(s)$ (see Fig. \ref{fig:velocity}).
\end{conv}

\begin{figure}
\centering
\includegraphics[width=1\textwidth]{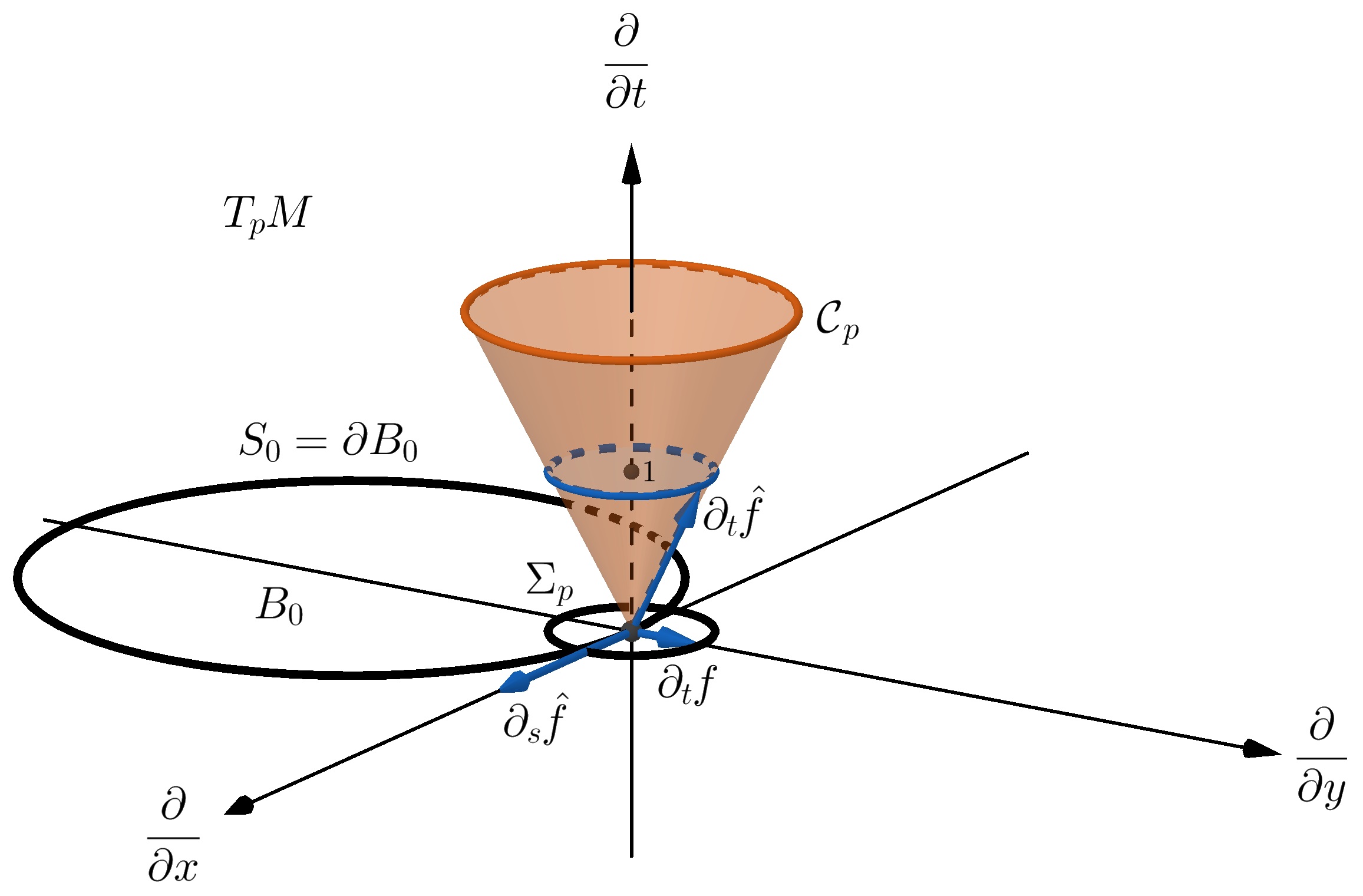}
\caption{The case of dimension $ n = 2 $ and $ r = 1 $. In the tangent space $ T_pM $ of $ M $ at $ p = (0,s) \in S_0 $, the cone $ \mathcal{C}_p $ establishes the lightlike directions. The lightlike vector $ \partial_t\hat{f} = (1,\partial_tf) $ is the velocity of the spacetime trajectory of the wave that heads out from $ B_0 $, being its projection $ \partial_tf \in \Sigma_p $ the velocity of the wave through the space $ N $. The tangent direction to $ S_0 $ is marked by $ \partial_s\hat{f} $ (but its ``length'' plays no role). Recall $ \partial_t\hat{f} \bot_G \partial_s\hat{f}$, i.e., $\partial_t f \bot_F \partial_s f$.}
\label{fig:velocity}
\end{figure}

In order to calculate this wavemap and its wavefronts, our Thm. \ref{th:geod_gen} can be particularized to give an ODE solution.
 
\begin{cor}
\label{th:geod_wildfires}
For each $ s \in S $, the wavemap $ \hat{f}(t,s) = (t,f(t,s))$ $ = (x^0,x^1(t,s),x^2(t,s)) $ of a wildfire is given by the following ODE system:
$$
\partial_t^2x^k = -\gamma_{\ ij}^k(\partial_t\hat{f})\partial_tx^i\partial_tx^j + \gamma_{\ ij}^0(\partial_t\hat{f})\partial_tx^i\partial_tx^j\partial_tx^k, \quad  k=1,2.
$$
Therefore, the spatial trajectories of the fire are the solutions $ f(t,s) = (x^1(t,s),x^2(t,s)) $ whose initial conditions satisfy
\begin{itemize}
\item $ f(0,s) = (x_0^1(s),x_0^2(s)) = s $,
\item $ \partial_tf(0,s) = (\partial_tx^1(0,s),\partial_tx^2(0,s)) $ is the unique $ F $-unit vector $ F $-orthogonal to $ \partial_sf(0,s) $ and pointing outwards from $B_0$.
\end{itemize}
\end{cor}

However, a PDE approach has been studied in the literature \cite{R,M16,M17}. In our framework, we can naturally obtain an equivalent PDE system in the spacetime. Then, this can be formulated as a purely ``spatial'' solution in terms of the (time-dependent) Finsler metric $F$.

\begin{thm}
\label{th:ort_cond}
For some $\varepsilon>0$, the wavemap $\hat f$ of a wildfire is characterized in $[0,\varepsilon)\times S$ by $ \hat f(0,s)=(0,s)\in S_0 $ and the following equivalent conditions (expressible as PDE's):  
\begin{itemize}
\item Orthogonality conditions in terms of $ G $: 
\begin{equation}
\label{eq:ort_cond_G}
\left\lbrace{
\begin{array}{l}
t \mapsto \hat{f}(t,s) = (t,f(t,s)) \ \text{\text{is a lightlike curve pointing outwards}},\\
\partial_t\hat{f}(t,s) \bot_G \partial_s\hat{f}(t,s).
\end{array}
}\right.
\end{equation}
\item Orthogonality conditions in terms of $ F $:
\begin{equation}
\label{eq:ort_cond_F}
\left\lbrace{
\begin{array}{l}
F(\partial_tf(t,s))  = 1, \ \text{\textit{with}} \ \partial_tf(t,s) \ \text{\textit{pointing outwards}},\\
\partial_tf(t,s) \bot_F \partial_sf(t,s).
\end{array}
}\right.
\end{equation}
\end{itemize}
\end{thm}
\begin{proof}
Let us verify first that $ \hat{f} $ satisfies the stated conditions. By the definition of $ \hat{f} $, the first condition in \eqref{eq:ort_cond_G} holds everywhere and the second one for $t=0$. By the definition of the fronts, any first-arriving trajectory must also be time-minimizing between each two fronts and, so, orthogonal to the fronts by Cor. \ref{cor:zermelo0}. Thus, the required $\varepsilon$ follows from Thm. \ref{cor:zermelo}. The PDE system expression is straightforward (see \eqref{eq:rich1}, \eqref{eq:rich2} below for the explicit one in the case of elliptic indicatrices). For the equivalence between these conditions and \eqref{eq:ort_cond_F}, recall \eqref{eq:equiv_G_F}.

Conversely, as we have just proven that the wavemap is a solution of the orthogonality conditions, it is enough to see that, in fact, it is the only solution. Let $ \hat f $ be an arbitrary solution of \eqref{eq:ort_cond_G} and denote by $ \gamma_s(t) $ and $ \beta_t(s) $ the corresponding longitudinal and transversal curves $ t \mapsto \hat f(t,s) $ and $ s \mapsto \hat f(t,s) $, resp. Our aim is to prove that $ \gamma_s $ is a cone geodesic (i.e., lightlike pregeodesic) for every $ s\in S $ because then, the uniqueness of $ t $-parametrized cone geodesics with the same initial conditions implies that $ \gamma_s $ coincides with the wavemap from $ s $, i.e., $ \hat f $ is the wavemap. Therefore, it is enough to prove that $ D_{\gamma_s}^{\gamma'_s}\gamma'_s = \lambda\gamma'_s $ for some real function $ \lambda $, where $ D^{\gamma'_s}_{\gamma_s} $ denotes the covariant derivative (associated with the Chern connection) along $ \gamma_s $ having $ \gamma'_s $ as a reference vector (see \cite{J13} for background and notation). Note that $ \gamma_s(t) $ can be regarded as a variation, being $ J_s(t) \coloneqq \partial_s\hat f(s,t) $ the corresponding variational vector field, which is nonzero for $ t $ smaller than some $ \varepsilon'>0 $ (a posteriori, this $ \varepsilon' $ can be assumed to be equal to the $ \varepsilon $ found in the first paragraph of the proof). By \eqref{eq:ort_cond_G}, $ \gamma'_s $ is lightlike and $ g^G_{\gamma'_s}(\gamma'_s,J_s)=0 $, so $ J_s(t) $ and $ \gamma'_s(t) $ generate the tangent plane $ T_{\gamma'_s(t)}\mathcal{C}_{\gamma_s(t)} $, which is degenerate in the direction of $ \gamma'_s $. Therefore, the condition $ D_{\gamma_s}^{\gamma'_s}\gamma'_s = \lambda\gamma'_s $ is equivalent to
$$
\left\lbrace
\begin{array}{l}
g^G_{\gamma'_s}(D_{\gamma_s}^{\gamma'_s}\gamma'_s,\gamma'_s)=0,\\
g^G_{\gamma'_s}(D_{\gamma_s}^{\gamma'_s}\gamma'_s,J_s)=0.
\end{array}
\right.
$$
To prove the first equation, note that $ g_{\gamma'_s}(\gamma'_s,\gamma'_s)=0 $ and therefore, using the almost $ g^G $-compatibility of the Chern connection \cite[Eq. (4)]{J13},
$$
0=\frac{d}{dt}g^G_{\gamma'_s}(\gamma'_s,\gamma'_s)=2g^G_{\gamma'_s}(D_{\gamma_s}^{\gamma'_s}\gamma'_s,\gamma'_s)+2C_{\gamma'_s}(D_{\gamma_s}^{\gamma'_s}\gamma'_s,\gamma'_s,\gamma'_s)=2g^G_{\gamma'_s}(D_{\gamma_s}^{\gamma'_s}\gamma'_s,\gamma'_s),
$$
where the term in the Cartan tensor $ C_{\gamma'_s} $ vanishes by homogeneity, as it is evaluated repeateadly in $ \gamma'_s $ \cite[Eq. (2)]{J13}. To prove the second one, recall that $ D_{\beta_t}^{\gamma'_s}\gamma'_s = D_{\gamma_s}^{\gamma'_s}\beta'_t $ (see \cite[Prop. 3.2]{J13}), so
\begin{equation}
\nonumber
\begin{split}
0 = & \frac{d}{ds}g^G_{\gamma'_s}(\gamma'_s,\gamma'_s)=2g^G_{\gamma'_s}(D_{\beta_t}^{\gamma'_s}\gamma'_s,\gamma'_s)+2C_{\gamma'_s}(D_{\beta_t}^{\gamma'_s}\gamma'_s,\gamma'_s,\gamma'_s)\\
= & 2g^G_{\gamma'_s}(D_{\gamma_s}^{\gamma'_s}\beta'_t,\gamma'_s)=2g^G_{\gamma'_s}(D_{\gamma_s}^{\gamma'_s}J_s,\gamma'_s)=-2g^G_{\gamma'_s}(J_s,D_{\gamma_s}^{\gamma'_s}\gamma'_s)
\end{split}
\end{equation}
(for the last equality, take $ t $-derivatives in $ g^G_{\gamma'_s}(J_s,\gamma'_s)=0$).
\end{proof}

\begin{rem}\label{r_PDEyMarkvorsen}
The orthogonality conditions in terms of $F$ \eqref{eq:ort_cond_F} are the ones Markvorsen arrives at in \cite[Cor. 7.4]{M16} (time-independent case) and \cite[Thm. 4.4]{M17} (time-dependent case) using a Lagrangian (Finslerian or rheonomic) on the space.

However, the spacetime interpretation provides not only a neat proof of the uniqueness of solution to \eqref{eq:ort_cond_F}, but also a more accurate result. Indeed, the characterization of $\hat f$ holds for all the points $(t,s)$ with $t < c(s)$, as well as in the points with $t=c(s)$ by continuity. If such a point $(c(s),s)$ is not a focal point for $G$, then the characterization can be extended to a neighborhood of it. Nevertheless, if $(c(s),s)$ is a focal point then $\partial_sf(t_0,s_0)$ will vanish and the second orthogonality condition of each pair will give no information beyond it.
\end{rem}

\section{Ellipsoids and quadratic simplification}\label{s5}
The simplest analytical anisotropic approximation to the propagation of the wave occurs when at each $p\in M$, the field of velocities $\Sigma_p$ is an ellipsoid, not necessarily centered at the origin,  which includes the case of Richards' model for wildfires \cite{R}.
The ellipsoidal character of $\Sigma$ implies that the corresponding cone structure $\C$ will be compatible with a classical Lorentz metric $g$, apart from the Lorentz-Finsler one $G=dt^2-F^2$. So, although our computation of the wavemap $\hat f$ applies to this specific case, next $\hat f$ will also be computed by means of the geodesics of $g$. This widely simplifies the equations \eqref{eq:geod_t_gen} because the formal Christoffel symbols $\gamma^k_{ij}$ in \eqref{eq:christ_formal} will become the Christoffel ones $\Gamma^k_{ij}$ of the Lorentz metric $g$. So, they will depend only on the point $p=(t,x)$ but not on the direction, skipping the Finslerian entanglement of $F$. From a technical viewpoint, we take into account and develop further the stationary-to-Randers correspondence in \cite{CJS11}.

\subsection{Trajectories using a classical Lorentz metric}
Consider a hypersurface $ \Sigma^0 \subset \textup{Ker}(dt) $ of centered ellipsoids $ \Sigma^0_p $ varying smoothly with $ p \in M $ (i.e., $ \Sigma^0 $ is transverse to the fibers of $ \text{Ker}(dt) $) and a smooth section $ W $ of $ \textup{Ker}(dt)$ (i.e., $W$ is a time-dependent vector field on $N$) so that
$$
\Sigma \coloneqq \Sigma^0 + W = \lbrace v+W_{\pi^M(v)}: v \in \Sigma^0 \rbrace.
$$
We will refer to $W$ as the {\em wind}, which represents any physical phenomenon that generates a displacement on the propagation, usually associated with the medium where the wave propagates. Indeed, $ W $ can represent the wind if the wave propagates through the air, but also water streams if the propagation takes place in the sea or in a river, or other phenomena. As stated in \S \ref{subsec:general_setting}, we will assume that the wind is ``mild'', which means that the zero section lies in the (open) region enclosed by $ \Sigma $ (the unit ball of $ F $); this guarantees that $ \Sigma $ properly defines a Finsler metric of Randers type for each $t$. For all $ p \in M $, the vectors $ v \in \Sigma^0_p $ are characterized by the ellipsoid equation $ Q_p(v) = 1 $, i.e.,
$$\Sigma^0_p = \lbrace v \in \textup{Ker}(dt_p): Q_p(v) = 1 \rbrace,
$$
where $ Q_p $ is a quadratic form that depends on the orientation of the ellipsoid and its semi-axes. $ Q_p $ determines a norm $ H_p \coloneqq \sqrt{Q_p} $ in $ \textup{Ker}(dt_p) $ that satisfies the parallelogram law, so it induces a Euclidean scalar product
$$
h_p(v,u) \coloneqq \frac{1}{4}(Q(v+u) - Q(v-u)), \quad \forall v,u \in \textup{Ker}(dt_p),
$$
being $ \Sigma^0_p $ its indicatrix. Since the ellipsoids vary smoothly from point to point, one has a Riemannian metric $ h $ on $ \textup{Ker}(dt) $. If $ \Sigma $ is the indicatrix of the Finsler metric $ F $, given by the displaced ellipses, then
\begin{equation}\label{eZermeloData}
h_p\left( \frac{v}{F_p(v)}-W_p , \frac{v}{F_p(v)}-W_p \right) = 1, \quad \hbox{that is,} \quad H_p(v-F_p(v)W_p) = F_p(v),
\end{equation}
for any $ p \in M,  v \in \textup{Ker}(dt_p) $ (recall $ H_p(v) = \sqrt{h_p(v,v)}$). The pair $(h,W)$ is the {\em Zermelo data} for the (time-dependent) Randers metric $F$. The formula \eqref{eZermeloData} characterizes them and the constraint $h(W,W)<1$ is implicit in the assumption of mild wind. These are the elements to construct the required Lorentzian metric\footnote{Recall that signature $(+,-\dots ,-)$ is used here, in contrast with \cite{CJS11,CJS}.} $g$.

\begin{prop}
\label{lem:lorentz}
Let $\Sigma$ be determined by Zermelo data $(h,W)$ as above. Its  cone structure $\C$ (associated with
$ G = dt^2-F^2 $) is also the cone structure of the Lorentz metric $ g = \Lambda dt^2 - 2\omega dt - h $, where $ \Lambda \coloneqq 1-h(W,W) $ and $ \omega \coloneqq h(\cdot,-W) $.
\end{prop}
\begin{proof}
To check that $ G $ and $ g $ share the same lightlike vectors, observe that
\begin{equation}\label{e_26}
\begin{split}
& G(\tau,v) = 0 \Leftrightarrow \tau^2 = F(v)^2 \Leftrightarrow \tau^2 = H(v-\tau W)^2 \\
& \Leftrightarrow \tau^2 = h(v,v) - 2\tau h(v,W) + \tau^2h(W,W) \Leftrightarrow g((\tau,v),(\tau,v)) = 0
\end{split}
\end{equation}
for any $ (\tau,v) \in T_pM, p \in M $, using \eqref{eZermeloData} and the definition of $g$.
\end{proof}

Now, we can  proceed as in \S \ref{subsec:cone_geod} and obtain the wavemap by solving the geodesic equation, with the (direction-independent) Christoffel symbols
\begin{equation}
\label{eq:christ_g}
\Gamma_{\ ij}^k (t,x) = \frac{1}{2} g^{kr}\left( \frac{\partial g_{rj}}{\partial x^i} + \frac{\partial g_{ri}}{\partial x^j} - \frac{\partial g_{ij}}{\partial x^r} \right)(t,x), \quad i,j,k = 0,\ldots,n
\end{equation}
(using Notation \ref{not:einstein}), where at each $(t,x)\in M$,  putting $h_{ij}=h(\frac{\partial}{\partial x^i},\frac{\partial}{\partial x^j})$,
\begin{equation}\label{eq:g}
\{g_{ij}\} =
\left(\begin{array}{cccc}
\Lambda & -\omega(\frac{\partial}{\partial x^1}) & \cdots & -\omega(\frac{\partial}{\partial x^n})\\
-\omega(\frac{\partial}{\partial x^1}) & -h_{11} & \cdots & -h_{1n}\\
\vdots & \vdots & \ddots & \vdots\\
-\omega(\frac{\partial}{\partial x^n}) & -h_{n1} & \cdots & -h_{nn}
\end{array}\right).
\end{equation}

\begin{lemma}\label{lem:666}
Any $u\in \Sigma_{(0,s)}\cap S_0^{\perp_F}$ (i.e. $u$ is $F$-unit and $F$-orthogonal to $S_0$, so that $\hat u=(1,u)\in \nu(S_0)_{(0,s)}$) can be written as 
\begin{equation}\label{e_266}
u=v+W_{(0,s)}, \quad \hbox{where} \quad h(v,v)=1, \ v \in S_0^{\perp_h},
\end{equation}
i.e., $v\in \text{Ker}(dt_{(0,s)}) (\equiv T_sN)$ is $h$-unit and $h$-orthogonal to $S_0$.
\end{lemma}
\begin{proof}
Putting $v \coloneqq u-W_{(0,s)}$, $F(u)=1$ is equivalent to $h(v,v)=1$ from \eqref{e_26}. Now, observe that for any $ w \in T_sN $, $ u \bot_F w $ means that $ w $ is tangent to the indicatrix $ \Sigma $ at $ u $, or equivalently, $ w $ is tangent to $ \Sigma^0 = \Sigma-W $ at $ u-W_{(0,s)} $, i.e., $ u-W_{(0,s)} \bot_h w $.
\end{proof}

So, the wavemap $ \hat{f}(t,s,u=v+W_{(0,s)})$ can be regarded as a function of $(t,s,v)\in [0,\infty)\times S\times \mathds{S}^{r-1}$, where $v$ satisfies \eqref{e_266}. This function will be called the {\em Lorentzian wavemap} and denoted with the same letter, that is, $ \hat{f}(t,s,v) (=(t,f(t,s,v))=(t,x^1(t,s,v),\ldots,x^n(t,s,v)))$, with no harm.

\begin{thm} \label{t5.3}
Assume that the propagation of the wave is given by a Randers metric on $ \textup{Ker}(dt) $ with Zermelo data $ (h,W) $. For each $(s,v)\in S\times \mathds{S}^{r-1}$ (identifiable to the $h$-unit orthogonal bundle to $S_0$ in $ \{t=0\} $), the Lorentzian wavemap $\hat f(t,s,v) = (x^0,x^1(t,s,v),\ldots,x^n(t,s,v))
$ is characterized by the following ODE system:
\begin{equation}
\label{eq:geod_t}
\partial_t^2x^k = -\Gamma_{\ ij}^k\partial_tx^i \partial_tx^j + \Gamma_{\ ij}^0\partial_tx^i \partial_tx^j\partial_tx^k, \quad k=1,\ldots,n,
\end{equation}
(see \eqref{eq:christ_g}, \eqref{eq:g} and Prop. \ref{lem:lorentz}) with the initial conditions
\begin{itemize}
\item $ f(0,s,v) = (x^1_0(s,v),\ldots,x^n_0(s,v)) = s $,
\item $ \partial_tf(0,s,v) = v+W_{(0,s)}$. 
\end{itemize}
\end{thm}
\begin{proof}
Following the same procedure as in Thm. \ref{th:geod_gen} but working with the Lorentzian metric $ g $ and its Christoffel symbols, one directly obtains \eqref{eq:geod_t}. For the initial conditions, just recall Lem. \ref{lem:666}.
\end{proof}

\subsection{Richards' equations for wildfires}
The well-known Richards' equations correspond to the PDE's \eqref{eq:ort_cond_F}  in our Thm. \ref{th:ort_cond}, where $\Sigma$ lies in the ellipsoidal case and it is determined by Zermelo data $(h,W)$, with $ h(W,W)<1 $. For the sake of completeness, we will consider  them  in our framework. Recall that the setting in \S \ref{subsec:wildfires} and \S \ref{subsec:orth_cond} applies; in particular, $n=2$, $r=1$ and $S_0$ is a closed curve (Conv. \ref{conventionfire}). Also, we will use here the notation $ (x,y) \coloneqq (x^1,x^2) $.

\begin{prop}
In the case of wildfires with Zermelo data $(h,W)$, the orthogonality conditions \eqref{eq:ort_cond_F}  become equivalent to
\begin{equation}
\label{eq:ort_cond_1}
1 = h(W,W)-2h(\partial_tf,W)+h(\partial_tf,\partial_tf),
\end{equation}
\begin{equation}
\label{eq:ort_cond_2}
0=h(\partial_tf-W,\partial_sf),
\end{equation}
with $ \partial_tf $ pointing outwards.
\end{prop}
\begin{proof} Just recall that by Lem. \ref{lem:666}, 
\begin{equation}
\nonumber
F(\partial_tf)=1 \Leftrightarrow h(\partial_tf-W,\partial_tf-W)=1,
\end{equation}
\begin{equation}
\nonumber
\partial_tf \bot_F \partial_sf \Leftrightarrow \partial_tf-W \bot_h \partial_sf.
\end{equation}
\end{proof}

Fixing the natural basis $ \{ \frac{\partial}{\partial x}\rvert_p, \frac{\partial}{\partial y}\rvert_p \} $ on $ \textup{Ker}(dt_p) \equiv \R^2 $ and working with coordinates, write $W_p=(w_1(p),w_2(p))$. The quadratic form for a centered ellipse with semi-axes $ a(p),b(p) $ rotated an angle $ \theta(p) $\footnote{Note that $ a,b $ and $ \theta $, as well as $w_1, w_2$, depend on $ p = (t,x,y) $ so that, in particular, they may vary over time. In the following equations this dependence will be assumed implicitly in order to avoid clutter.} in the clockwise direction is
$$
Q_p(v_1,v_2) = \left(\frac{v_1\cos\theta-v_2\sin\theta}{a}\right)^2 + \left(\frac{v_1\sin\theta+v_2\cos\theta}{b}\right)^2,
$$
for all $ v = (v_1,v_2) \in \textup{Ker}(dt_p)$. Thus, the matrix associated with $ h_p $ is
\begin{equation}\label{eq:h_theta}
\{h_{ij}\} = \frac{1}{a^2 b^2}
\left(\begin{array}{cc}
a^2\sin^2\theta+b^2\cos^2\theta & (a^2-b^2)\sin\theta\cos\theta \\
(a^2-b^2)\sin\theta\cos\theta & a^2\cos^2\theta+b^2\sin^2\theta
\end{array}\right).
\end{equation}

\begin{thm}[Richards' equations]
The wavemap $ \hat{f}(t,s) = (t,f(t,s)) = (t,x(t,s),y(t,s)) $ of a wildfire determined by Zermelo data $(h,W)$ with $h$ in \eqref{eq:h_theta} and $ W= (w_1,w_2) $, $ h(W,W)<1 $, is characterized by the following PDE system for $ \partial_sf(t,s) = (\partial_sx(t,s),\partial_sy(t,s)) $, $ \partial_tf(t,s) = (\partial_tx(t,s),\partial_ty(t,s)) $:
\begin{equation}
\label{eq:rich1}
\partial_tx = \pm \frac{a^2\cos\theta(\partial_sx\sin\theta+\partial_sy\cos\theta)-b^2\sin\theta(\partial_sx\cos\theta-\partial_sy\sin\theta)}{\sqrt{a^2(\partial_sx\sin\theta+\partial_sy\cos\theta)^2+b^2(\partial_sx\cos\theta-\partial_sy\sin\theta)^2}} + w_1,
\end{equation}
\begin{equation}
\label{eq:rich2}
\partial_ty = \pm \frac{-a^2\sin\theta(\partial_sx\sin\theta+\partial_sy\cos\theta)-b^2\cos\theta(\partial_sx\cos\theta-\partial_sy\sin\theta)}{\sqrt{a^2(\partial_sx\sin\theta+\partial_sy\cos\theta)^2+b^2(\partial_sx\cos\theta-\partial_sy\sin\theta)^2}} + w_2,
\end{equation}
where the $ \pm $ has to be chosen in order for $ \partial_tf(t,s) $ to point outwards,\footnote{Namely, choose $ + $ for $S_0$  counter-clockwise parametrized, and $ - $ otherwise.} and the initial condition $ {f}(0,s) = (x_0(s),y_0(s)) $ (identifiable to $s$) holds.
\end{thm}
\begin{proof}
The proof is a straightforward computation taking into account that, from \eqref{eq:ort_cond_2}, $\partial_t f-W$ must be $ h $-orthogonal to $\partial_sf$, its length is controlled by \eqref{eq:ort_cond_1} and it is oriented outwards (recall also Thm. \ref{th:ort_cond} for the characterization).

Anyway, computations can be simplified by taking into account that, first, when $ \theta = 0 $, \eqref{eq:ort_cond_1} and \eqref{eq:ort_cond_2} yield:
\begin{equation}
\label{eq:rich1_0}
\partial_tx = \pm \frac{a^2 \partial_sy}{\sqrt{a^2 (\partial_sy)^2 + b^2 (\partial_sx)^2}} + w_1,
\qquad \partial_ty = \pm \frac{-b^2 \partial_sx}{\sqrt{a^2 (\partial_sy)^2 + b^2 (\partial_sx)^2}} + w_2,
\end{equation} 
in agreement with \eqref{eq:rich1}, \eqref{eq:rich2}. Then, in order to obtain $ (\partial_tx,\partial_ty) $ for arbitrary  $ \theta $, rotate the coordinate axes the same angle $ \theta $, use \eqref{eq:rich1_0}, and recover the initial coordinates.
\end{proof}

\begin{rem} 
To check that these equations agree with Richards' \cite{R}, notice that we have assumed independence between $h$ and $W$ in the model. However,  Richards  originally claimed as an experimental fact (within certain limits) that  the ratio $ a/b $ depends on the wind speed only, with the vector $ W $ always aligned with the major axis of the ellipse. This means that the wind not only displaces the initial indicatrices, but also deforms them. In the simplest model (used by Richards), the fire spread can be approximated in the tangent space by a sphere, which is deformed to an ellipse when the wind appears. With this model in mind, it is convenient to write the components $(c_1,c_2)$ of the wind vector $ W=(w_1,w_2) $ with respect to the (orthonormal) basis defined by the main axes of the ellipses. Namely, put
$$
\left(\begin{array}{c}
c_1 \\
c_2
\end{array}\right)
= R_{\theta}^{-1}
\left(\begin{array}{c}
w_1 \\
w_2
\end{array}\right)
\;
\hbox{so that}
\;
\left\lbrace\begin{array}{l}
w_1 = c_1\cos\theta + c_2\sin\theta\\
w_2 = -c_1\sin\theta + c_2\cos\theta,
\end{array}\right. 
$$
where $ R^{-1}_{\theta} $ is the counter-clockwise rotation matrix, and substitute $(w_1,w_2)$ in our equations \eqref{eq:rich1} and \eqref{eq:rich2}. This way, one arrives exactly at 
Markvorsen's equations (who also considered a wind independent of the metric), see
\cite[Thm. 9.1]{M16} and \cite[Thm. 8.1]{M17}. Then, Richards' equations \cite[Eqs. (10), (11)]{R} are obtained by setting $c_2=0$, so that the wind blows along the semi-axis $ a $.
\end{rem}

\section{The case of strong wind}\label{s6}
Until now we have assumed that the wind is mild, which guaranteed that $ \Sigma $ is the indicatrix of a Finsler metric. Now we will take a step further by allowing a strong wind 
(i.e., some zero vectors may lie outside $ \Sigma $), so that a type of 
wind Finslerian structure (\S \ref{6.0} below) is obtained. In principle, this does not affect the geometric framework of the cone structure. However, the physical interpretation of the wavemap will depend on the type of the wave and the specific situation at hand. 

For example, when modeling sound waves in a medium that moves faster than the sound itself, the wavefront is given effectively by the $\C$-achronal boundary  $\partial J^+(S_0)$, while $J^+(S_0)$ provides the particles of the medium affected by the wave at each time. In \S \ref{6.1} we will see that $\partial J^+(S_0)$ can be computed just by changing $T=\partial/\partial t$ in our cone triple. However, the  problem of first-arriving trajectories also changes, and the wind Finslerian structure must be used to recover the original one (see Rem. \ref{r6.1}). As this problem is  essential for wildfires, it is studied specifically in \S \ref{6.2}, including an estimate of the burned area by the active front (Rem. \ref{rem6.3}).  

\begin{figure}
\centering
\includegraphics[width=1\textwidth]{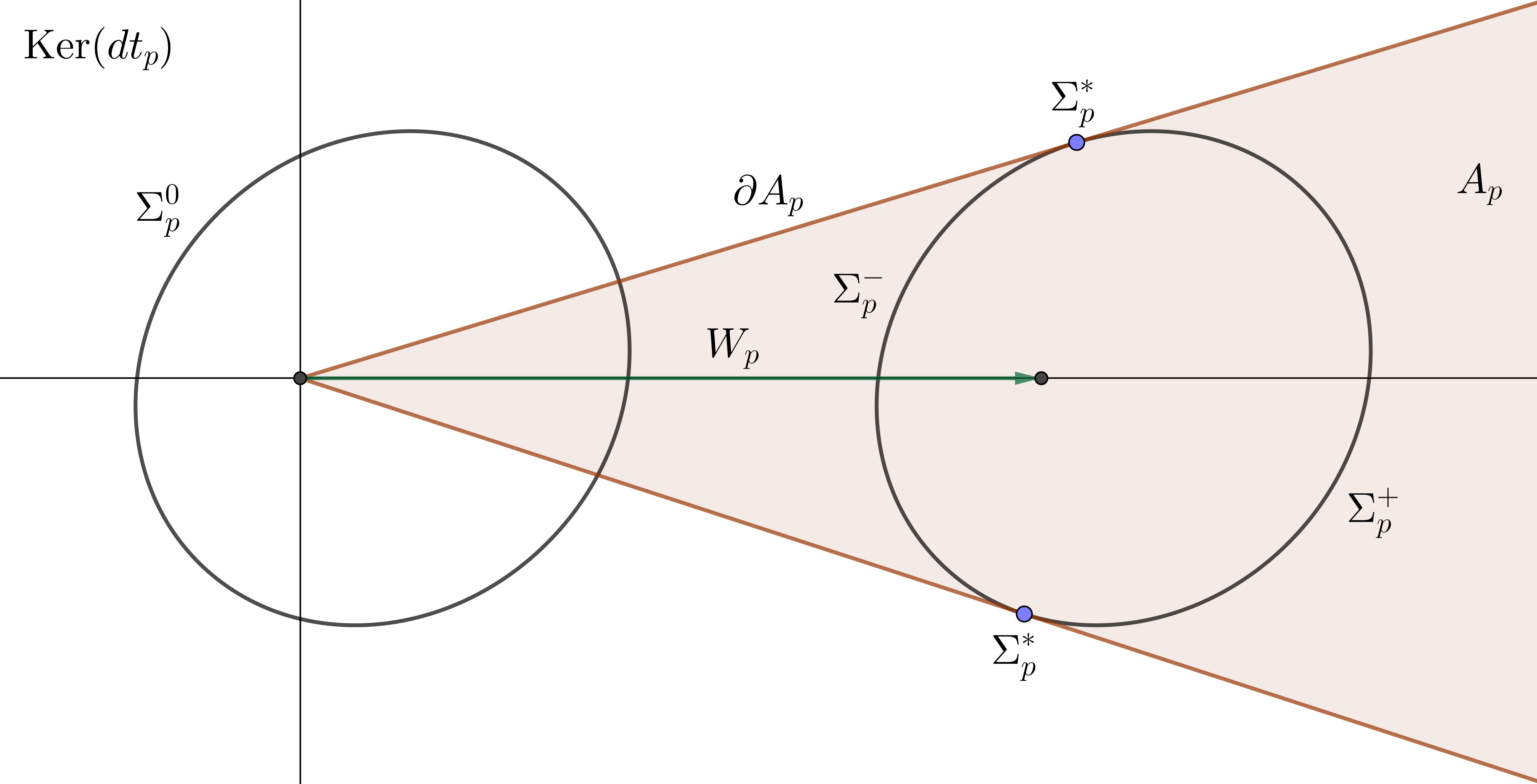}
\caption{The case of strong wind: the initial indicatrix $ \Sigma^0_p $ is displaced by the wind $ W_p $ in such a way that the zero vector no longer lies inside $ \Sigma_p=\Sigma^0_p+W_p $. The tangent lines to $ \Sigma_p $ from the origin divide it into three disjoint components: $ \Sigma_p = \Sigma_p^+ \cup \Sigma_p^- \cup \Sigma^*_p $. $ \Sigma_p^+ $ is the indicatrix of a conic Minkowski norm $ F_p $, while $ \Sigma_p^- $ defines a Lorentzian norm $ (F_l)_p $. These norms are defined in the conic domain $ A_p $ but can be extended continuously to $ \overline{A}_p \setminus \{0\} $ in such a way that if $ v \in \Sigma^*_p $, then $ F_p(v)=(F_l)_p(v)=1 $.}
\label{fig:conic_metric}
\end{figure}

\subsection{Wind Finslerian setting}\label{6.0}
As in \S \ref{sec:general_setting}, we start with a cone structure $ \mathcal{C}$, constructed from the velocities of the wave, and put $\Sigma = (\mathcal{C} \cap dt^{-1}(1))-\frac{\partial}{\partial t}$. So far,  the restriction  $\partial/\partial t$ timelike ({\em mild wind} case) and the triple $ (\Omega = dt,T = \partial/\partial t,F)$ have been used; 
next, this restriction is removed ({\em arbitrary wind} case) and we work directly with the {\em cone wind triple} $ (\Omega = dt,T = \partial/\partial t,\Sigma )$. The latter can be regarded as a wind Finslerian structure (in the sense of \cite[Def. 2.8]{CJS}) varying smoothly with the time. The regions where $\partial/\partial t$ is non-causal (resp.  lightlike, timelike) are called of {\em strong} (resp. {\em critical, mild}) wind. In the region of strong wind, $ \Sigma $ determines  
an (open, connected) conic  domain $ A\subset \text{Ker}(dt)$ whose radial half-lines  intersect transversely $\Sigma$. In this domain, we have both  a
conic Finsler metric $ F $ and a Lorentzian Finsler one $ F_l $,  with indicatrices, resp., the convex  and concave (from infinity)  portions $\Sigma^+, \Sigma^-$ of $\Sigma$  (see Fig. \ref{fig:conic_metric}). Clearly, $ F < F_l $ on $ A $ and both metrics can be continuously extended to $ \overline{A} \setminus \textbf{0} $ (both extensions agree on $ \partial A \setminus \textbf{0} $, see \cite[Prop. 2.12]{CJS} and Fig. \ref{fig:conic_metric}).
 At the points where the wind is critical, the zero vector lies on $ \Sigma $ and $A$ becomes an open half space;  when it is mild, $A= \hbox{Ker}(dt)\setminus \mathbf{0}$ and  $ F $ becomes the already studied Finsler metric (in these two cases $ F_l $ is not defined or can be regarded as equal to $\infty$).  

\subsection{Description of $\C$ and associated Zermelo's problem}\label{6.1} 
When $\C$ is given by a cone wind  triple $(dt,T,\Sigma)$, one can recover a cone triple by  choosing any timelike vector field $ T^0 $ such that $ dt(T^0) = 1 $ and taking the corresponding Finsler metric $ F^0 $ as in Thm. \ref{th:cone_triple}. Then, $ (dt,T^0,F^0) $  becomes a cone triple for $ \mathcal{C} $ on a new decomposition of $M$ as a product $\R\times N$.
 Specifically, choose a vector field ({\em wind}) $W$  with $dt(W)\equiv 0$ so that the zero section lies pointwise inside $\Sigma^0 \coloneqq \Sigma - W $ and put  $ T^0 = \partial/\partial t + W$. Then, $T^0$  becomes  timelike, the indicatrix of the searched Finsler metric is $\Sigma^0= \Sigma - W= (\mathcal{C} \cap dt^{-1}(1)) - T^0  $ and the flow of $T^0$ provides a new decomposition $M\ni p \rightarrow (l,x)\in \R\times N$ so that $T^0=\partial/\partial l$. The integral curves of  $T$ and $T_0$ represent, resp., initial observers at rest and particles of the medium, the latter moving with velocity $W$ with respect to the former, while $\Sigma$ and  
$\Sigma^0$ contain the propagation velocities of the wave with respect to $T,T^0$.

The Lorentz-Finsler metric associated with $ (dt,T^0=\partial/\partial t+W,F^0) $ is $ G(\tau,v) = \tau^2-F^0(v-\tau W)^2 $, where $ (\tau,v) = \tau \frac{\partial}{\partial t}\rvert_p+v = \tau T^0_p+(v-\tau W_p) $, with 
$ v \in \textup{Ker}(dt_p) $, $ p \in M $.
Thus, for any $ \hat v = (1,v) \in T_pM \setminus \text{Span}(T^0_p), \hat w = (0,w) \in T_pM $,
$$
\hat{v} \in \mathcal{C}_p \Leftrightarrow G(\hat{v}) = 0 \Leftrightarrow F^0(v-W)=1 \Leftrightarrow F(v) = 1 \text{ or } F_l(v) = 1
$$
and, in this case,
\begin{equation}
\label{eq:G_F_strong_wind}
\begin{split}
& \hat v \bot_{G} \hat w \Leftrightarrow (v-W) \bot_{F^0} w \Leftrightarrow w \in T_{(v-W)}\Sigma^0 \Leftrightarrow w \in T_v\Sigma \\
& \Leftrightarrow v \bot_F w, \quad \text{or} \quad v \bot_{F_l} w, \quad \text{or} \quad v \parallel w \text{ with } v \in \partial A,
\end{split}
\end{equation}
where the last possibility comes from the fact that the fundamental tensors $ g^F_v $ and $ g^{F_l}_v $ are not defined when $ v\in \partial A $ (compare with \eqref{eq:equiv_G_F}).
Consequently, we can work with the new triple $ (dt,\partial/\partial t+W,F^0) $ and obtain the wavemap (and thus the wavefront at any time) in exactly the same way as we did in Thm. \ref{th:geod_gen}. Indeed, the only change in this theorem is that, now, the initial condition $\partial_t f(0,s,u)=u$ is such that $u-W \in \Sigma^0_{(0,s)}\cap S_0^{\perp_{F^0}}$ (since now $ \nu(S_0) $ is the set of $ dt $-unit vectors $ \hat{u}=(1,u) $ with $ u-W $ as above; see Fig. \ref{fig:strong_wind}). 


\begin{figure}
\centering
\includegraphics[width=1\textwidth]{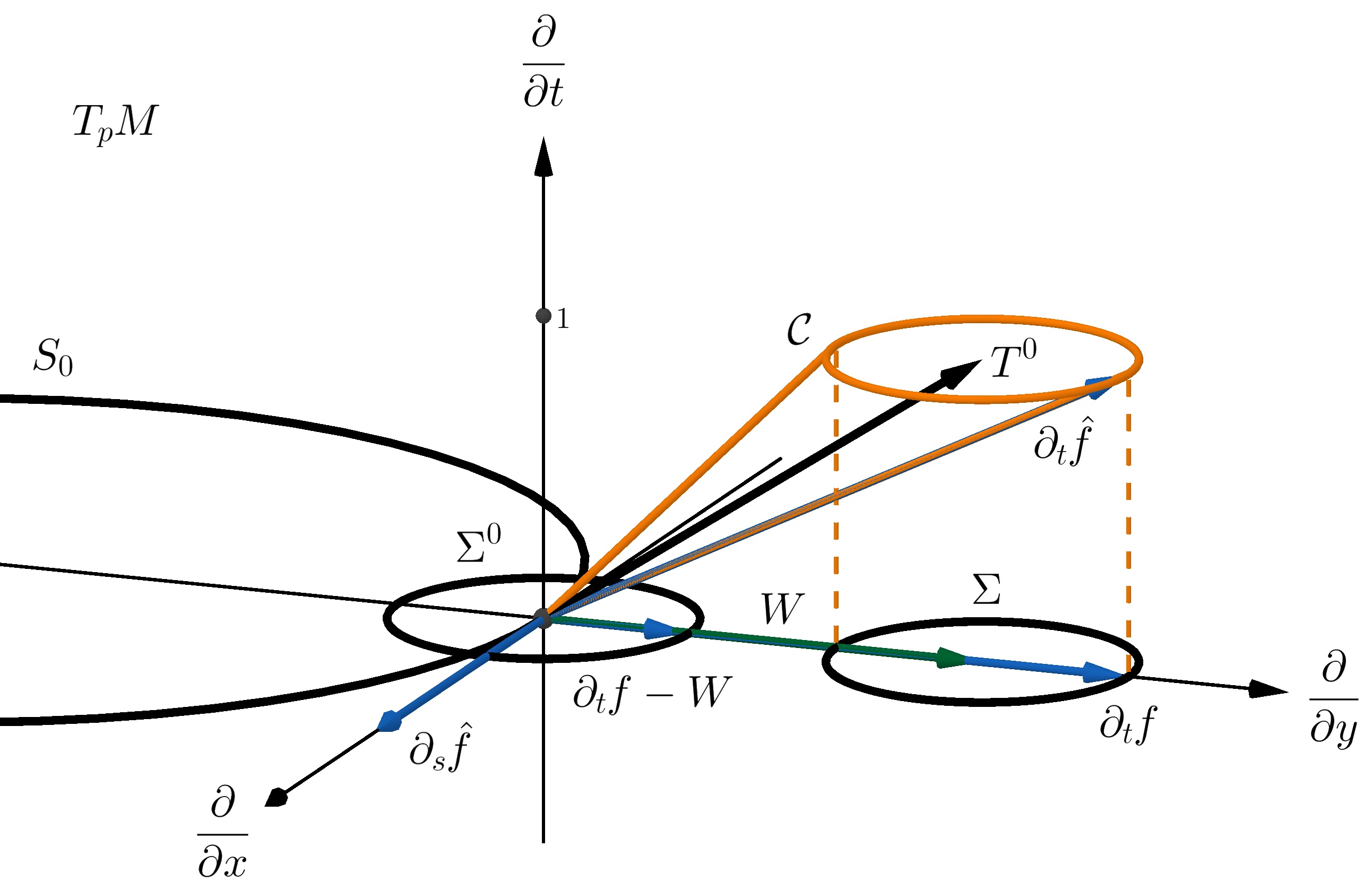}
\caption{The spacetime setting in the case of dimension $ n=2 $ and $ r=1 $. The wavemap $ \hat f(t,s) $ is defined so that $ t \mapsto \hat f(t,s) $ is the unique cone geodesic $ t $-parametrized with initial velocity $ \partial_t \hat f(0,s) = (1,\partial_tf(0,s)) \in \nu(S_0) $ and such that $ \partial_tf(0,s)-W_{(0,s)} $ points out to the exterior of $ S_0 $. This means that $ \partial_t\hat{f} $ is lightlike, $ F^0(\partial_tf-W) = 1 $ (i.e., $ \partial_tf \in \Sigma $),$ \partial_t\hat{f} \bot_{G} \partial_s\hat{f} $ and $ \partial_tf-W \bot_{F^0} \partial_sf $. Compare with Fig. \ref{fig:velocity}.}
\label{fig:strong_wind}
\end{figure}

\begin{rem}\label{r6.1} Using $ T^0=\partial/\partial t+W$, the computation of $\partial J^+(S_0)$ is reduced to the case when the wind is mild, but the property of being first-arriving for the cone geodesics occurs with respect to the integral curves of $ T^0$ (observers co-moving with the medium).  
 $ \partial J^+(S_0) $ still contains the outermost trajectories of the wave but, in general, they are no longer first-arriving from $ S_0 $  with respect to the original  $T=\partial/\partial t$ (i.e., the Zermelo problem changes, see Fig. \ref{fig:cones_strong_wind}).  
Indeed, the cone geodesics that minimize the original arrival time are those $ G $-orthogonal to $ S_0 $ whose projection is $ F $-orthogonal to $ S_0 $ (and necessarily $ F $-unit),  where $F$ is now the (time-dependent) conic Finsler metric provided by the
wind Finslerian structure. Recall that $F_l $-orthogonal geodesics maximize the arrival time, according to the interpretation in the wind Finslerian case \cite[Prop. 2.41]{CJS} (compare also with \cite[Cor. 6.18, Thm. 7.8]{CJS}). 

Once $\partial J^+(S_0)$ has been computed as above, the arrival time of the wave to an observer in $T=\partial/\partial t$ can be be obtained from the intersection of the corresponding integral curve of $T$ and $\partial J^+(S_0)$. From a practical viewpoint, for any $\hat w =(1,w) \in \nu(S_0)_{(0,s)}$ with $w \in A_{(0,s)}$ and $F(w)=1$, the corresponding geodesic $\hat \gamma$ of the metric $G^{\text{conic}}=dt^2-F^2$ (defined on the conic domain $\cup_{p\in M} \hbox{Span}(T_p)\times A_p$) gives the first-arriving trajectory (at least for small times) in the spatial direction $w$.
\end{rem}

\begin{figure}
\centering
\includegraphics[width=1\textwidth]{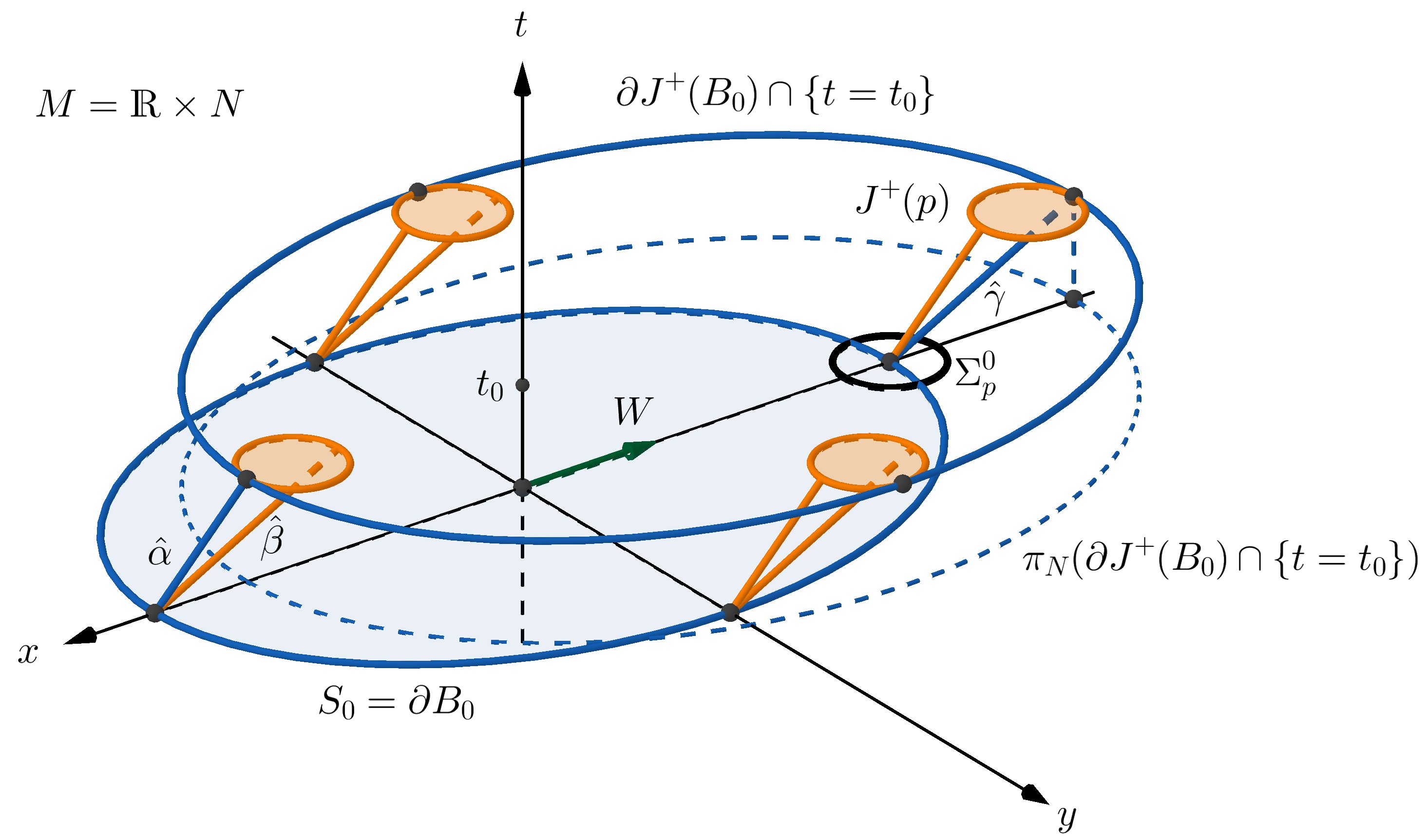}
\caption{Similar case to Fig. \ref{fig:cones} but with a constant strong wind $ W $. The choice of $ \partial_tf-W $ pointing outwards makes $ \partial J^+(B_0) $ contain cone geodesics whose projection is $ F $-unit (such as $ \hat \gamma $) and also $ F_l $-unit cone geodesics (such as $ \hat \alpha $). Note that the latter are not first-arriving: e.g., $ \hat \beta $ arrives earlier than $ \hat \alpha $ at every point. Also, these curves enter the initial region $ B $, so $ \partial J^+(B_0) $ might not represent a realistic wavefront in some situations, such as wildfires.}
\label{fig:cones_strong_wind}
\end{figure}

\subsection{Active firefront of a wildfire}\label{6.2}
A wildfire with strong wind can be modelled roughly with a wind Finslerian structure $\Sigma$, which represents the velocities of propagation (in absence of wind), displaced by the wind $W$. This is not physically accurate, as the effect of the wind is not simply a displacement.\footnote{A more realistic model is developed in \cite{JPS}.} However, it can be used to obtain a rough estimate of its active front of propagation.
Recall that in such a model $\partial J^+(B_0)$ contains trajectories that enter the already burned initial area (e.g., $ \hat \alpha $ in Fig. \ref{fig:cones_strong_wind}). Obviously, these curves must not be considered as trajectories of the firefront. In fact, we can assume that the fire is extinguished when compelled to enter an already burned area (this might also not be physically realistic in some cases but, anyway, it will not be relevant for our rough estimate). So, the following modification of the framework would be applicable.

 Given the initial source of the fire $ S_0=\partial B_0 $, the wavemap $\hat f(t,s)=(t,f(t,s)), (t,s)\in [0,\infty) \times S$ can be defined as in \S \ref{subsec:orth_cond} taken into account that, now, the choice of one of the two lightlike directions at each $\nu(S_0)_p$ must ensure that $ \partial_tf-W $ points outwards from $B_0$ at $p$. In order to identify the points where the fire is extinguished we give the following proposition.
 
\begin{prop}
Let $ (0,s) \in S_0 $. The vector $ \partial_tf(0,s) $ is parallel to $ \partial_sf(0,s) $ if and only if $ \partial_tf(0,s) \in \partial A_{(0,s)} $ (i.e., $ \partial_tf(0,s) $ is both $ F $-unit and $ F_l $-unit). In this case we say that $ (0,s) $ is an {\em extinction point of} $ S_0 $ (see Fig. \ref{fig:wildfires_strong_wind}).
\end{prop}
\begin{proof}
Recall that $ \partial_t\hat f(0,s) \bot_G \partial_s\hat f(0,s) $, which means that $ \partial_sf(0,s) $ is tangent to $ \Sigma $ at $ \partial_tf(0,s) $, as stated in \eqref{eq:G_F_strong_wind}. Then $ \partial_tf(0,s) || \partial_sf(0,s) $ only holds when $ \partial_tf(0,s) \in \partial A_{(0,s)} $.
\end{proof}

\begin{figure}
\centering
\includegraphics[width=1\textwidth]{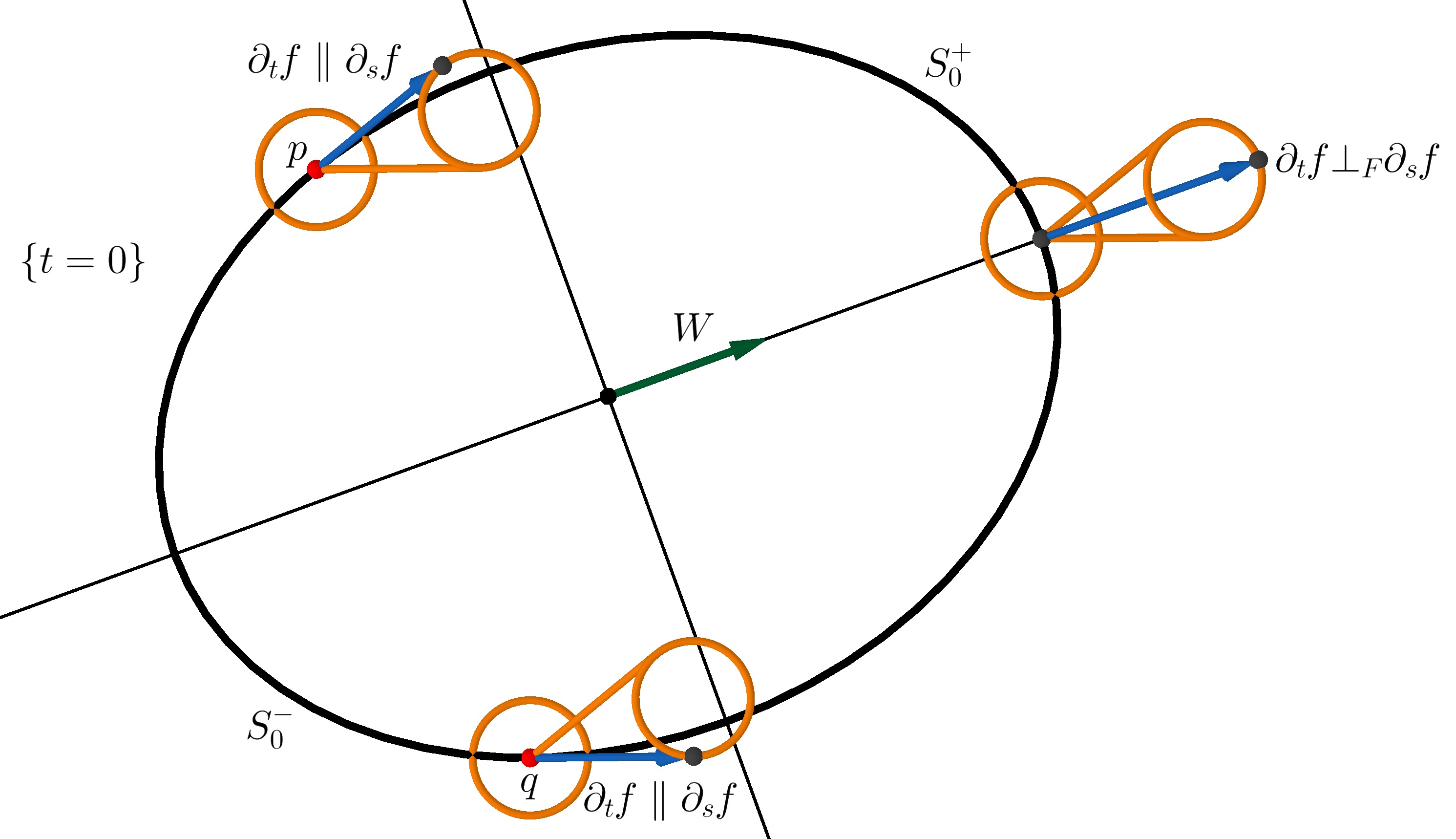}
\caption{The case of Fig. \ref{fig:cones_strong_wind} from an aerial view. In this situation, $ p $ and $ q $ are the extinction points, i.e., the only points on $ S_0 $ where $ \partial_tf \parallel \partial_sf $ or, equivalently, $ \partial_tf \in \partial A_p $. These points make a separation between the region $ S_0^+ $ that provides the active firefront, and the region $ S_0^- $ where the fire ends up extinguished. The orthogonality condition $ \partial_tf \bot_F \partial_sf $ holds on $ S_0^+ $.}
\label{fig:wildfires_strong_wind}
\end{figure}

The extinction points divide $ S_0 $ into several open connected components (typically two if $ S_0 $ is convex and the wind does not change dramatically from one point to another), alternating regions where the spatial trajectories of the fire are $ F $-unit and go outwards ({\em trajectories of the active firefront}) and regions where these trajectories are $ F_l $-unit and go inwards (which will be discarded). Let $ S_0^+=\{0\}\times S^+ $ and $ S_0^- $ be the union of the former and latter type of connected components, resp. (excluding the extinction points). Then, the {\em  active firefront} at each $t_0$ is $\{\hat f(t_0,s): s\in S^+\}$, i.e., the component of $ \partial J^+(S_0^+) $ that heads out from $ B_0 $ (intersected with $\{t=t_0\}$). Recall that each one of its $t-$parametrized cone geodesics $\hat \gamma(t)=(t,\gamma(t))$ remains first-arriving from $S_0$, at least for small time. Indeed, $\gamma'(0)$ becomes unitary for the conic Finsler metric $F$ and Lem. \ref{th:zermelo} is still applicable (even though Thm. \ref{cor:zermelo} cannot be reobtained as $ S_0^+ $ is not compact). The active firefront can also be obtained from the ODE  in Cor. \ref{th:geod_wildfires}: simply, change $S$ by $S^+$ recalling that $F$ is now the conic Finsler metric.\footnote{On the contrary, as the fire on $ S_0^- $ can be regarded as extinguished, the wavemap $\hat f$ does not represent the physical firefront on these points.}

\begin{rem}
\label{rem6.3}
From this rough model, the estimate of the burned area until the time $t_0>0$ would be  (the closure of) $  f([0,t_0],S^+) $.
 Summing up, when  the wind is strong, the  active firefront is determined just by the (time-dependent) conic Finsler metric $F$ of the wind Finslerian structure and only its $F$-unit directions should be taken into account for the computation of the wavemap. This provides a seemingly highly singular description of this front, but the overall spacetime viewpoint restores smoothness (in the spirit of \cite{CJS}). 
\end{rem}

\section*{Acknowledgments}
This work is a result of the activity developed within the framework of the Programme in Support of Excellence Groups of the Regi\'on de Murcia, Spain, by Fundaci\'on S\'eneca, Science and Technology Agency of the Regi\'on de Murcia. MAJ was partially supported by MICINN/FEDER project reference PGC2018-097046-B-I00 and Fundaci\'on S\'eneca (Regi\'on de Murcia) project reference 19901/GERM/15, Spain, and EPR and MS by Spanish MINECO/FEDER project reference MTM2016-78807-C2-1-P. MS was also partially supported by FEDER-Andaluc\'ia grant A-FQM-494-UGR18, and EPR by Programa de Becas de Iniciaci\'on a la Investigaci\'on para Estudiantes de M\'asteres Oficiales de la U. Granada, Spain.

\end{document}